\documentclass[12pt,reqno, oneside]{amsart}
\makeatletter
\@namedef{subjclassname@2020}{%
  \textup{2020} Mathematics Subject Classification}
\makeatother
\usepackage{amsmath,amssymb,latexsym,esint,cite,mathrsfs,slashed, dsfont}
\usepackage{verbatim,cite,relsize,color,soul}
\usepackage[latin1]{inputenc}
\usepackage{microtype}
\usepackage{color,enumitem,graphicx, xcolor}
\usepackage{mathtools}
\mathtoolsset{showonlyrefs}
\usepackage[colorlinks=true,urlcolor=blue, citecolor=red,linkcolor=blue,
linktocpage,pdfpagelabels, bookmarksnumbered,bookmarksopen]{hyperref}
\usepackage[hyperpageref]{backref}
\usepackage[english]{babel}
\usepackage{geometry}

\numberwithin{equation}{section}
\newtheorem{theorem}{Theorem}[section]
\newtheorem{lemma}[theorem]{Lemma}

\newtheorem{prop}[theorem]{Proposition}
\theoremstyle{definition}

\newtheorem{remark}[theorem]{Remark}
\newtheorem*{claim}{Claim}

\newtheorem{thm}{Theorem}
\newtheorem{coro}{Corollary}

\newcommand{\R}{\mathbb R}
\newcommand{\C}{\mathbb C}

\DeclareMathOperator{\Ima}{Im}
\DeclareMathOperator{\Rea}{Re}

\title[$L^2$-Scattering NLS with mixed nonlinearities]
{NLS with mass-subcritical combined nonlinearities: small mass $L^2$-scattering}

\author[J. Bellazzini, L. Forcella, V. Georgiev
]{Jacopo Bellazzini, Luigi Forcella, and Vladimir Georgiev}

\address[J. Bellazzini]{Department of Mathematics, University of Pisa, Largo Bruno Pontecorvo, 5, 56127, Pisa, Italy}
\email{jacopo.bellazzini@unipi.it}

\address[L. Forcella]{Dipartimento di Matematica, Universit\`a Degli Studi di Pisa, Largo Bruno Pontecorvo, 5, 56127, Pisa, Italy}
\email{luigi.forcella@unipi.it}

\address[V. Georgiev]{Dipartimento di Matematica, Universit\`a Degli Studi di Pisa, Largo Bruno Pontecorvo, 5, 56127, Pisa, Italy, and Faculty of Science and Engineering, Waseda University, 3-4-1, Okubo, Shinjuku-ku, Tokyo 169-8555, Japan, and Institute of Mathematics and Informatics,  Bulgarian Academy of Sciences,  Acad. Georgi Bonchev Str., Block 8, 1113 Sofia, Bulgaria}

\email{vladimir.simeonov.gueorguiev@unipi.it}



\begin{document}
	
\begin{abstract}
We prove small data scattering in the mass-subcritical regime for the NLS equation with double nonlinearities, where a focusing leading term is perturbed by a lower order defocusing nonlinear term. Our proof relies on the pseudo-conformal transformation in conjunction with a general variational argument used to obtain the positivity of certain modified energies. Moreover, the smallness assumption is only on the mass of the initial data, and not on the whole $\Sigma$-norm.
\end{abstract}

\maketitle

\section{Introduction}
In this paper, we consider the Cauchy problem for the  nonlinear Schr\"odinger equation (NLS) with combined nonlinearities
\begin{equation}\label{NLS0}
    \begin{cases}
    i \partial_t \psi +\Delta_x \psi =  |\psi|^{q-1}\psi -  |\psi|^{p-1}\psi\\
     \psi(0)= \psi_0  \in \Sigma
    \end{cases},
\end{equation}
where $\psi=\psi(t,x)$, $\psi:\mathbb{R}\times \R^d \mapsto \mathbb{C}$,  $\Delta_x$ is the standard Laplace operator with respect to the space variables $x$,  and 
\begin{equation*}\label{MA1}
    1+\frac{2}{d}< q < p < 1+ \frac{4}{d}.
\end{equation*}
 The space $\Sigma$ is defined as the subspace of $H^1(\mathbb R^d)$ with finite variance:
\[
\Sigma=\{u\in H^1(\mathbb R^d)\, \hbox{ s.t. } \, |x|u \in L^2(\mathbb R^d) \},
\]
endowed with norm $\|f\|_{\Sigma}^2=\|f\|_{H^1(\mathbb R^d)}^2+\||x|u\|_{L^2(\mathbb R^d)}^2$. 
Taking into account the signs in front of the nonlinear terms in \eqref{NLS0}, we are dealing with a mixed nonlinearity  with  focusing  mass-subcritical term of order $p$, and a defocusing term of order $q \in (1+\frac{2}{d},p)$. The critical power $p_0=1+\frac4d$ is defined by borrowing from the usual scaling invariant equation with one nonlinearity
\[
i\partial_t \psi+\Delta_x \psi = \pm |\psi|^{p-1}\psi,
\] and refers to the conservation of the $L^2$-norm of a solution under the scaling invariance of the equation given by $\psi_\lambda(t,x)=\lambda^{\frac{2}{p_0-1}}\psi(\lambda^2 t, \lambda x)$. \\
The exponent $q_0=1+\frac{2}{d}$ is instead known as the exponent separating the so-called short-range and long-range  nonlinearities. See the discussion below. 

A function $\psi(t)$, with  $\psi\in  C((-T_{\min}, T_{\max});\Sigma)$ is said  a mild solution    to \eqref{NLS0} if it satisfies the integral equation
\begin{equation*}\label{NLSf4} 
    \psi(t) = U(t) \psi_0 -i \int_0^t U(t-s) g(\psi(s)) d s,
\end{equation*}
where $U(t)=e^{it\Delta_x}$ is the free Schr\"odinger propagator
and $g(\psi) = |\psi|^{q-1}\psi - |\psi|^{p-1}\psi$.

At a formal level, equation \eqref{NLS0} preserves the following quantities: the mass, 
\begin{equation}\label{eq:mass}
    M(\psi(t)) = \int_{\mathbb{R}^d} |\psi(t,x)|^2 dx,
\end{equation}
 the energy
\begin{equation}\label{eq:en}
\begin{aligned}
     E(\psi(t))& = \frac{1}{2}\int_{\mathbb{R}^d} |\nabla \psi(t,x)|^2 dx + \frac{1}{q+1} \int_{\mathbb{R}^d} |\psi(t,x)|^{q+1} dx\\
     &- \frac{1}{p+1} \int_{\mathbb{R}^d} |\psi(t,x)|^{p+1} dx,
\end{aligned}
\end{equation}
and the momentum
\begin{equation}\label{eq:m}
\begin{aligned}
     P(\psi(t))& = \Ima \int_{\R^d} \psi(t,x) \nabla \bar\psi(t,x) dx.
\end{aligned}
\end{equation}

A solution to \eqref{NLS0} satisfying conservation of  \eqref{eq:mass}, \eqref{eq:en}, and \eqref{eq:m} is known to exist. See Proposition \ref{prop:lwp} in Section \ref{Sec:pct}.

\subsection{Purpose of the work}  In this paper, we are interested in the scattering properties of solutions to \eqref{NLS0}. In particular, once it is  known that a solution exists globally-in-time, we ask whether the solution behaves linearly for large times. Specifically, one may ask the existence of functions $\psi^\pm$ such that 
\begin{equation}\label{def-scattering}
    \lim_{t\to\pm \infty} \|U(-t)\psi(t)-\psi^\pm\|_X=0,
\end{equation}
in some functional space $X$.
For our scopes, $X$ will be the space $L^2(\mathbb R^d)$, $H^1(\mathbb R^d)$, or $\Sigma$, and we will refer to  \eqref{def-scattering} as $L^2$-scattering, $H^1$-scattering, or $\Sigma$-scattering, respectively. Note that $U(t)$ is an isometry in $H^s(\mathbb R^d)$, $s\in\mathbb{R}$, and so \eqref{def-scattering} is equivalent to $\displaystyle\lim_{t\to\pm \infty} \|\psi(t)-U(t)\psi^\pm\|_{X}=0$ if $X=L^2(\mathbb{R}^d)$ or $X=H^1(\mathbb {R}^d)$. This is no more true when $X=\Sigma$, and we refer to \cite{Beg} for some results in this direction.  
Before going into details, we state our main result.  
\begin{thm}\label{mainthm-sca-0}
Let $1+\frac{2}{d}< q < p < 1+ \frac{4}{d}$. There  exists a positive mass $\rho^{\star}$ such that for any $\psi_0\in\Sigma$ with   $\|\psi_0\|_{L^2(\mathbb R^d)}<\rho^{\star}$ a solution to \eqref{NLS0} scatters in $L^2(\mathbb R^d)$.
\end{thm}

Let us give the following comments about the content of the theorem above. 
\begin{remark}
We emphasize that our smallness assumption is only on the mass of the initial datum, and not in the whole $\Sigma$-norm.  This is a major difference with respect to the classical scattering results as in \cite{CW92, GOV94, NO02} for the  NLS equation (see below for further details).  In 1D, we also cite \cite{HN15}, where the authors consider an arbitrary (complex) linear combination of nonlinear terms, and proved a version of the scattering with small data in some weighted Sobolev spaces. 
\end{remark}
\begin{remark}
Below,  we will provide an upper bound on the mass $\rho^\star$ by means of the threshold mass yielding the existence of ground states for  equation \eqref{eq}.
\end{remark}

It is worth briefly recalling what happens for standard NLS equation with one pure-power nonlinearity. 
The problem of scattering for the following defocusing equation
\begin{equation}\label{NLS-defoc}
i\partial_t \psi+\Delta_x \psi = |\psi|^{q-1}\psi
\end{equation}
and the focusing equation
\begin{equation}\label{NLS-foc}
i\partial_t \psi+\Delta_x \psi = -|\psi|^{p-1}\psi
\end{equation}
has a long history, and a huge amount of works has been produced since the 70's.
In the setting of the present paper, where mass-subcritical nonlinearities are considered, i.e., $1<q<1+\frac{4}{d}$, for the defocusing model \eqref{NLS-defoc} we mention the work by Tsutsumi and Yajima \cite{TY84}, who proved $L^2$-scattering in the whole short-range interval $\frac2d<q<\frac4d$  for solutions in $\Sigma$ by employing the pseudo-conformal transformation, see Section \ref{Sec:pct}. The results in \cite{TY84} are optimal as   Strauss \cite{Strauss74} and Barab \cite{Barab} proved failure of $L^2$-scattering in the long-range regime $1<q\leq \frac2d$. We also cite the more recent paper \cite{BGTV} by Burq, Tzvetkov, Visciglia, and the third author, in which the result in \cite{TY84}   is upgraded to the $H^1$-topology. For the $\Sigma$-scattering, we refer to \cite{CW92,T85}, where the scattering in the strongest norm $\Sigma$ is proved up to a restriction of the short-range interval, i.e., $q\in(q_1, \frac4d)$, with $q_1>q_0$.  \\

As for the focusing equation \eqref{NLS-foc}, the situation is more complicated due to the conflict between the dispersive effect of the linear part of the equation and the nonlinear interaction. While the energy is non-negative definite for \eqref{NLS-defoc},  it does not have a sign for \eqref{NLS-foc}, and when an initial datum is not small (in some sense), the nonlinear effect can be strong enough to prevent scattering.   This is the case of particular solutions, called standing waves, which in fact do exist for \eqref{NLS-foc}, and do not exist for \eqref{NLS-defoc}. Such solutions are of the form  $\psi(t,x)=e^{i\omega t}u(x)$,
 with $\omega\in \R$, and $u(x)\in \C$ is a time-independent function belonging to  $H^1(\R^d)$ which satisfies 
\begin{equation}\label{eq-sw0}
- \Delta u + \omega u -|u|^{p-1}u=0.
\end{equation}
 It is worth recalling that solutions to \eqref{eq-sw0} exist for any mass, and standing waves are global non-scattering solutions. 
 For scattering results under  smallness assumption in some weighted $L^2$-spaces, we refer to the classical works by Cazenave and Weissler, Ginibre, Ozawa, and Velo, and Nakanishi and Ozawa, see  \cite{CW92, GOV94, NO02}, respectively. We also mention the recent work by Ifrim and Tataru \cite{IT-24-Vietnam, IT-2025} for new results about the one dimensional, focusing, cubic NLS. The transition between scattering/non-scattering solutions in term of the size of the initial datum is studied in the papers by Masaki \cite{Satoshi15, Satoshi17}, where the author works in some weighted $L^2$-spaces. It is worth mentioning that in contrast to the mass-critical or mass-supercritical cases $p\geq \frac4d$, Masaki results show that the ground state does not play  the role of the threshold separating the scattering and non-scattering regimes. \\

As the existence of standing waves may represent an obstruction for scattering, we start by considering the possible existence of standing waves for  our mixed model \eqref{NLS0}. Specifically, we look at solutions to \eqref{NLS0} of the form $\psi(t,x)=e^{i\omega t}u(x)$,  with $\omega\in \R$ and $u(x)\in \C$ is a time-independent  $H^1(\R^d)$-function which satisfies 
\begin{equation}\label{eq}
- \Delta u + \omega u +|u|^{q-1}u-|u|^{p-1}u=0.
\end{equation}

Since the mass is a conserved physical quantity along the flow of \eqref{NLS0}, a natural approach to finding solutions $u$ to \eqref{eq} is to seek critical points of the energy functional, constrained to the $L^2$-spheres in $H^1(\mathbb{R}^d)$, which are defined by
\[S_{\rho}=\{u\in H^{1}(\mathbb R^d) \ : \  \|u\|_{ L^2(\mathbb R^d)}=\rho\}.
\]
Thus, a solution to \eqref{eq} is understood as a pair $(\omega_{\rho}, u_\rho) \in \mathbb{R} \times H^1(\mathbb{R}^d)$, where $\omega_{\rho}$ serves as the Lagrange multiplier corresponding to the critical point  $u_{\rho}$ on the constraint set $S_{\rho}$.\medskip

For a fixed a mass $\rho>0$, we define the \emph{ground state energy} the quantity $\mathrm{I}_{\rho^2}$, defined as the infimum of the energy functional over $S_\rho$, namely, 
\begin{equation*}\label{mini1}
\mathrm{I}_{\rho^2}=\inf_{S_{\rho}} E(u),
\end{equation*}
where $E$ is defined as in \eqref{eq:en}.

We have the following existence and non-existence result for the existence of ground states. 

\begin{thm}\label{mainthmforE}
Let $d\geq 1$ and  $1<q< p<1+\frac{4}{d}$. Then there exists a strictly positive mass $\rho_{0,E}$ such that:\\
\textup{(i)} $ \mathrm{I}_{\rho^2}=0$  for all  $\rho\in (0,\rho_{0,E}]$;\\
\textup{(ii)} $ \mathrm{I}_{\rho^2}<0$  for all  $\rho\in (\rho_{0,E}, \infty)$.\smallskip

\noindent Moreover, there are no constrained minimizers for $0<\rho<\rho_{0,E}$, and  for all  $\rho\in [\rho_{0,E}, \infty)$ there  exists $u_{\rho}\in S_{\rho}$ such that $\mathrm{I}_{\rho^2}=E(u_{\rho})$. 
\end{thm}
 We will refer to the minimal mass $\rho_{0,E}$ as the ground state threshold mass.
 \begin{remark}
    Note that any ground state $u_\rho$ given in Theorem \ref{mainthmforE} belongs to $\Sigma$ as well. Indeed, it is a solution to the elliptic equation \eqref{eq} with $\omega>0$ and therefore is in $L^2(|x|^2dx)$. See Appendix \ref{app}.
\end{remark}

A natural question that arises after proving the existence of a threshold mass for the existence of ground states, is whether there exists a smaller threshold mass for the existence of generic standing waves. In order to answer this question, we recall that for any  standing wave $\psi(t,x)=e^{i\omega t}u(x)$ we have $G(u)=0$, where
 $G(u)$ is the following Pohozaev functional
\[ G(u)=\int_{\mathbb R^d}|\nabla u|^{2}dx+\frac{d(q-1)}{2(q+1)}\int_{\mathbb R^d}| u|^{q+1}dx-\frac{d(p-1)}{2(p+1)}\int_{\mathbb R^d}|u|^{p+1}dx.
\]
Note that  the Pohozaev functional has the same structure as the energy functional \eqref{eq:en}, the two being different only  by the  constants appearing in front of the integrals. Then,  it becomes relevant  to study   the existence of minimizers for a general defocusing-focusing energy functional of the type 
\[
E^{\alpha, \beta, \gamma}(u):=\alpha\|\nabla u\|_{L^2(\mathbb R^d)}^2+\beta \|u\|^{q+1}_{L^{q+1}(\mathbb R^d)}-\gamma\|u\|^{p+1}_{L^{p+1}(\mathbb R^d)}
\]
constrained to the manifold $S_{\rho}$, with positive constants $\alpha,\beta, \gamma$ and $1<q<p<1+\frac{4}{d}$.
Thus, we consider  the general problem 
\[
\mathrm{I}_{\rho^2}^{\alpha, \beta,\gamma}:=\inf_{S_{\rho}}E^{\alpha, \beta,\gamma}(u),
\]
and we prove the following existence and non-existence result of minimizers. 
\begin{thm}\label{mainthm}
Let $d\geq 1$, $1<q< p<1+\frac{4}{d}$, and $\alpha>0$, $\beta>0$, $\gamma>0$. Then there exists a strictly positive mass 
\begin{equation}\label{threshold-mass}
\rho_0=\rho_0(\alpha, \beta, \gamma)
\end{equation} such that:\\
\textup{(i)} $\mathrm{I}_{\rho^2}^{\alpha, \beta,\gamma}=0$  for all  $\rho\in (0,\rho_0]$;\\
\textup{(ii)} $ \mathrm{I}_{\rho^2}^{\alpha, \beta,\gamma}<0$  for all  $\rho\in (\rho_0, \infty)$.\smallskip

\noindent Moreover, there are no constrained minimizers for $0<\rho<\rho_0$, and  for all  $\rho\in [\rho_0, \infty)$ there  exists $u_{\rho}\in S_{\rho}$ such that $\mathrm{I}_{\rho^2}^{\alpha, \beta, \gamma}=E^{\alpha, \beta, \gamma}(u_{\rho})$. 
\end{thm}

It is clear that Theorem \ref{mainthmforE} is just a specific case 
of Theorem \ref{mainthm} when $\alpha=\frac 12, \beta=\frac{1}{q+1}, \gamma=\frac{1}{p+1}$, namely $\rho_0=\rho_{0,E}$, and that Theorem \ref{mainthm} also guarantees  that $G(u)>0$ for $u\in S_{\rho}$ when 
 $\rho<\rho_{0,SW}$ where 
\begin{equation}\label{defenergsw}
\rho_{0,SW}=\rho_0\left(1, \frac{d(q-1)}{2(q+1)}, \frac{d(p-1)}{2(p+1)}\right).
\end{equation}
As a consequence, we have the following.
\begin{coro}\label{cor1}
Standing waves solutions $\psi(t,x)=e^{i\omega t}u(x)$,  $\omega\in\mathbb R$ and $u\in H^1(\R^d)$ with mass smaller than $\rho_{0,SW}$ cannot exist.
\end{coro}

\begin{remark}
It is worth mentioning that a  general result for nonhomogeneous nonlinearities that gives an analogous statement of Theorem \ref{mainthm} is contained in  \cite{JL22}. In this paper, we give a new alternative and shorter proof that uses only scaling arguments and  is suitable for pure-power mixed nonlinearities. 
In this way, on the one hand,  all the results in this paper are self-contained. On the other hand, and most importantly, our ideas are precisely tailored to work on the dynamical properties of solutions to \eqref{NLS0}, as they enable us to perform suitable controls on various energy functionals that appear throughout the paper. We introduced this new approach very recently in the context of the Half-Wave equation, see \cite{BF2025}, and in this article we further develop our method to tackle the much harder scattering problem related to \eqref{NLS0}. 
\end{remark}
\begin{remark}
    Since the seminal work by Tao, Vi\c san, and Zhang  \cite{TVZ}, the NLS with combined nonlinearities has attracted a lot of attention. Concerning  the existence of standing waves when the leading nonlinearity is mass super-critical it is worth mentioning at least the following recent works \cite{BFG23, FH21, LRN20, SoaveJDE, SoaveJFA, JJTV}, and references therein. 
\end{remark}

Going back to our main goal, i.e., the long-time dynamics of solutions to \eqref{NLS0}, and having in mind Theorem \ref{mainthmforE} and Corollary \ref{cor1}, it is clear that the long-time dynamics of solutions to \eqref{NLS0} shall be rather different from those of the NLS equation with a single power-like nonlinear term. Indeed, we are in a scenario where standing waves may exist or not, depending on the mass.  We can re-state Theorem \ref{mainthm-sca-0} in a more precise and quantitative way. To this aim, we introduce the following exponents:
\begin{equation}\label{def:alphas}
\delta(q) = \frac{4-d(q-1)}{2}\quad \hbox{ and } \quad \delta(p) =\frac{4-d(p-1)}{2}.
\end{equation}
Note that both  exponents above are in the range $(0,1)$.

\begin{thm}\label{mainthm-sca}
Let $d\geq1$, $1+\frac{2}{d}< q < p < 1+ \frac{4}{d}$, and $\rho_{0,E}>0$ be the threshold mass for the existence of ground states given in Theorem \ref{mainthmforE}. Then  there  exists a positive mass $\rho^{\star}$,  given by 
\[
\rho^{\star}=\rho_0\left(\frac{1}{2},\frac{(1-\delta(q))(1-\delta(p))^{-\frac{\delta(q)}{\delta(p)}}}{q+1}, \frac{1}{p+1}\right)<\rho_{0,E},
\]
such that for any $\psi_0 \in \Sigma$ with $\|\psi_0\|_{L^2(\mathbb R^d)}<\rho^{\star}$ there exist $\psi_{\pm}\in L^2(\mathbb R^d)$ satisfying
\[
 \lim_{t\to\pm \infty}\|\psi(t)-U(t)\psi_{\pm}\|_{L^2(\mathbb R^d)}= 0\]
with $\psi(t)$ solution to \eqref{NLS0}.
\end{thm}
\begin{remark}
 In view of the content of Theorem \ref{mainthm-sca}, it is natural to ask whether the threshold mass $\rho^\star$ is optimal or not. Alternatively, we may ask if we can push $\rho^\star$ up to $\rho_{0,E}$ defined in Theorem \ref{mainthmforE}. We conjecture that this is not possible as we suspect the existence of standing waves which are not global minima for the energy functional, and have a mass strictly smaller than $\rho_{0,E}$. In the Appendix we show that the threshold mass $\rho_{0,E}$ defined in Theorem \ref{mainthmforE}, $\rho_{0,SW}$ defined in \eqref{defenergsw} and $\rho^{\star}$ defined in Theorem \ref{mainthmforE}, are ordered as $
\rho^{\star}<\rho_{0, SW}<\rho_{0,E}$.
\end{remark}

 Let us briefly illustrate the main ideas behind the proof of Theorem \ref{mainthm-sca}. 
First, we  introduce the pseudo-conformal transformation 
\[
\varphi(\tau,\xi) =    \psi\left(\frac{\tau}{1-\tau},\frac{\xi}{1-\tau}\right) (1-\tau)^{-d/2} e^{-\frac{i  |\xi|^2}{4(1-\tau)}}.
\]
 It is  known that  $e^{-it\Delta_x}\psi(t)$ has a strong limit in the $L^2$-topology as $t\to+\infty$ if and only if $\varphi(\tau)$ has a strong limit in the $L^2$-topology as $\tau\to1^-$, see Proposition \ref{sca-equi}. In the new variables, $\varphi(\tau)$ is a solution to the following NLS with time-dependent coefficients
\[
i \partial_\tau \varphi +  \Delta_\xi \varphi =  (1-\tau)^{-\delta(q)} |\varphi(\tau)|^{q-1}\varphi(\tau)
     - (1-\tau)^{-\delta(p)}|\varphi(\tau)|^{p-1}\varphi(\tau), 
\]
and to perform a Tsutsumi and Yajima argument \cite{TY84} it will be crucial to have suitable controls in terms of $(1-\tau$) for the quantities  $\|\nabla_\xi \varphi(\tau)\|_{L^2(\mathbb R^d)}$, $\|\varphi(\tau)\|_{L^{q+1}(\mathbb R^d)}$, and $\|\varphi(\tau)\|_{L^{p+1}(\mathbb R^d)}$. 
At this point, in order to control the aforementioned norms, we introduce the modified energy 
\begin{equation*}
    \begin{aligned}
        E_A(\tau,\varphi(\tau)) &:= \frac{(1-\tau)^A}{2}\|\nabla_\xi \varphi(\tau)\|_{L^2(\mathbb R^d)}^2
        +\frac{(1-\tau)^{A-\delta(q)}}{q+1} \|\varphi(\tau)\|_{L^{q+1}(\mathbb R^d)}^{q+1} \\
       & - \frac{(1-\tau)^{A-\delta(p)}}{p+1} \|\varphi(\tau)\|_{L^{p+1}(\mathbb R^d)}^{p+1}.
    \end{aligned}
\end{equation*}
The  modified energy above  is parametrized by a non-negative real number $A$ that satisfies $\delta(q)<A<1$ (which will be  chosen properly later  in the paper). Specifically, the modified energy $E_A$ is not conserved along the flow, so the  crucial point will be to exploit a bound on  the time evolution of the modified energy in terms of $E_A(0,\varphi(0))$, provided that the Cauchy datum has $L^2$-norm strictly lower than a certain mass $\rho_1(A)$.   The variational analysis developed in Theorem \ref{mainthm} will be the key tool for this purpose. Once the time evolution of the modified energy is controlled, we will get the estimates
\begin{equation*}
    \begin{aligned}
        \|\nabla_\xi \varphi(\tau)\|_{L^2(\mathbb R^d)}^2&\lesssim (1-\tau)^{-A},\\
\|\varphi(\tau)\|_{L^{q+1}(\mathbb R^d)}^{q+1}&\lesssim (1-\tau)^{\delta(q)-A},\\
\|\varphi(\tau)\|_{L^{p+1}(\mathbb R^d)}^{p+1}&\lesssim (1-\tau)^{\delta(p)-A}.
    \end{aligned}
\end{equation*}
that enable us to prove the long-time dynamical results by exploiting the  Tsutsumi and Yajima strategy. We therefore prove scattering below the threshold mass $\rho_1(A)$ parametrized by  $A$.
Eventually, we will show the monotonicity of the mass  $\rho_1(A)$ as $A\rightarrow 1^-$ and we compute the 
largest possible threshold mass given by 
\[
\rho^{\star}=\rho_0\left(\frac{1}{2},\frac{(1-\delta(q))(1-\delta(p))^{-\frac{\delta(q)}{\delta(p)}}}{q+1}, \frac{1}{p+1}\right).
\]

\begin{remark}
    To the best of our knowledge, this is the first $L^2$-scattering result in the full short-range regime with a smallness assumption only on the initial mass. It is worth mentioning that the problem of scattering has been widely treated in recent years when 
    the  nonlinear terms in \eqref{NLS0} (not necessarily in the defocusing-focusing case) are in the mass-supercritical and energy-subcritical/energy-critical interval, and we refer the reader to \cite{BDF24, AIKN13, Luo22, Luo24, Miao13, Miao16, Cheng16, Cheng20, AKN12, TVZ} and references therein. 
\end{remark}

\subsection{Notations} Along this paper, we use the following notation. The Lebesgue spaces $L^p(\mathbb{R}^d)$, $1\leq p< \infty$, are simply denoted by $L^p$ with the corresponding norms denoted by $\|f \|_p=(\int_{\mathbb{R}^d}|f(x)|^p dx)^{1/p}$. Since now on, we  omit the dependence on the whole space in the integrals. For $s\in\mathbb R$, $H^s(\mathbb{R^d}^d)=H^s=(1-\Delta)^{-s/2} L^2 $ are the $L^2$-based Sobolev space  endowed norms $\|f\|_{H^s}^2=\|(1-\Delta)^{s/2}f\|_{2}^2$. In particular, $\|f\|_{H^1}^2=\|\nabla f\|_2^2+\|f\|_2^2$. The $L^2$-pairing is defined as $
     \langle f , g \rangle = \int  f(x) \overline{g}(x) dx.$
 $X \lesssim Y$ stands for the usual shorthand notation for inequalities that involve implicit constants. Specifically, 
$X \lesssim Y$ means that $X \leq C Y$ for some positive constant $C$, and similarly for $X \gtrsim Y$. When both inequalities hold, we use $X\sim Y$. 

\section{The variational problem}
 
As described in the Introduction, our first goal is to find conditions for the existence of minimizers for a general defocusing-focusing energy functional of the type 
\begin{equation}\label{ener-general}
E^{\alpha, \beta, \gamma}(u):=\alpha\|\nabla u\|_{2}^2+\beta \|u\|^{q+1}_{q+1}-\gamma\|u\|^{p+1}_{p+1},
\end{equation}
constrained on the manifold $S_{\rho}:=\{ u\in H^1 \ : \ \|u\|_2=\rho\}$, with positive constants $\alpha,\beta, \gamma$ and $1<q<p<1+\frac{4}{d}$.
Therefore, we look at the problem 
\[
\mathrm{I}_{\rho^2}^{\alpha, \beta,\gamma}:=\inf_{S_{\rho}}E^{\alpha, \beta,\gamma}(u).
\]
We also introduce the Pohozaev functional 
\[
G^{\alpha, \beta, \gamma}(u):=2 \alpha\|\nabla u\|_{2}^2+\frac{d(q-1)}{2}\beta \|u\|^{q+1}_{q+1}-\frac{d(p-1)}{2}\gamma\|u\|^{p+1}_{p+1},
\]
and we recall that for all critical points of the energy functional $E^{\alpha, \beta, \gamma}(u)$ we have $G^{\alpha, \beta, \gamma}(u)=0.$
From now on, having fixed $\alpha, \beta, \gamma$, we simplify the notation for the reader's convenience using $E(u)$ instead of $E^{\alpha, \beta, \gamma}(u)$, $G(u)$ instead of $G^{\alpha, \beta, \gamma}(u)$, $\mathrm{I}_{\rho^2}$ instead of $\mathrm{I}_{\rho^2}^{\alpha, \beta,\gamma}$, and $\rho_0$ instead of $\rho_0(\alpha,\beta, \gamma)$.

First, we recall the following facts, for which we refer to the work by Lions \cite{Lions84}. 
\begin{lemma}\label{lemmainiz}
Consider the function $\rho\rightarrow \mathrm{I}_{\rho^2}$. The following properties hold true: \\
\textup{(i)}
 $\mathrm{I}_{\rho^2}$ is weakly subadditive, i.e., for any  $0<\mu<\rho$, 
\[\mathrm{I}_{\rho^2}\leq I_{\mu^2}+\mathrm{I}_{\rho^2-\mu^2};\]
\textup{(ii)} if $\mathrm{I}_{\rho^2}$ is strongly  subadditive, i.e., if  for any $0<\mu<\rho$
\[\mathrm{I}_{\rho^2}<I_{\mu^2}+\mathrm{I}_{\rho^2-\mu^2}, \]
then the infimum is attained. 
\end{lemma}

In order to prove Theorem \ref{mainthm}, we introduce the auxiliary problem \[\tilde E^{\alpha,\gamma}(u)=\tilde E(u):=\alpha\|\nabla u\|_{2}^2-\gamma\|u\|_{p+1}^{p+1},\]
and we define
\[
\mathrm{J}_{\rho^2}=\inf_{S_{\rho}}\tilde E(u).
\]

We start with the following negativity property of $\mathrm{J}_{\rho^2}$.
\begin{prop}\label{propomog}
For any $\rho>0$, $\mathrm{J}_{\rho^2}=\rho^{\frac{4(p+1)-2d(p-1)}{4+d-dp}}\mathrm{J}_1<0$. 
\end{prop}
\begin{proof}
Let us assume that $u$ has $L^2$-unitary norm. We claim that $\mathrm{J}_1<0$. Let us consider the mass-preserving scaling $u_{\lambda}=\lambda^{\frac d2}u(\lambda x)$ so that  $u_\lambda$ remains in $S_1$. We have that 
\[
\tilde E(u_{\lambda})= \alpha\lambda^2\|\nabla u\|_{2}^2-\gamma\lambda^{\frac d2(p-1)}\|u\|_{p+1}^{p+1}.
\]
As for $1<p<1+\frac 4d$ we have $\frac d2(p-1)<2$ and $\mathrm{J}_1\leq \tilde E(u_{\lambda})$ by definition, the claim follows by selecting $\lambda$ sufficiently small. 

\noindent Now, set $u_{\lambda}=\lambda^{\frac{2}{p-1}}u(\lambda x)$.
We then have that $\|u_{\lambda}\|_{2}^2=\lambda^{\frac{4}{p-1}-d}\|u\|_{2}^2=\lambda^{\frac{4}{p-1}-d}$.  By fixing $\lambda=\lambda(\rho) =\rho^{\frac{2(p-1)}{4-d(p-1)}}$
we have that $u_{\lambda}\in S_{\rho}$. For this choice of $\lambda$, the energy becomes $\tilde E(u_{\lambda(\rho)})=\lambda(\rho)^{\frac{2(p+1)-d(p-1)}{p-1}}\tilde E(u)=\rho^{\frac{4(p+1)-2d(p-1)}{4+d-dp}}\tilde E(u)$.
As the scaling map between  $S_1$ and $S_{\rho}$ is a bijection, $\mathrm{J}_{\rho^2}=\rho^{\frac{4(p+1)-2d(p-1)}{4+d-dp}}\mathrm{J}_1<0$, and the proof is done. 
\end{proof}
As a consequence, we can deduce the non-positivity of the function $\mathrm{I}_{\rho^2}$ and its decay property as $\rho\to0^+$.

\begin{prop}\label{propasynt}
$\mathrm{J}_{\rho^2}\leq \mathrm{I}_{\rho^2}\leq0$ and  $\displaystyle \lim_{\rho\rightarrow 0^+}\frac{\mathrm{I}_{\rho^2}}{\rho^2}=0$.
\end{prop}
\begin{proof}
The positivity of $\beta$ directly implies  $\mathrm{J}_{\rho^2}\leq \mathrm{I}_{\rho^2}$. 
Take now $u\in S_{\rho}$ and   $u_{\lambda}=\lambda^{\frac d2}u(\lambda x)$, so that $u_{\lambda}\in S_{\rho}$ for all $\lambda>0$. We have
\[E(u_{\lambda})=\alpha\lambda^2 \|\nabla u\|_{2}^2+\beta\lambda^{\frac d2 (q-1)}\|u\|_{q+1}^{q+1}-\gamma\lambda^{\frac d2(p-1)}\|u\|_{p+1}^{p+1},\]
then $E(u_{\lambda})\rightarrow 0$ for $\lambda\to0^+$, and hence  $ \mathrm{I}_{\rho^2}\leq 0$. 
 Proposition \ref{propomog} and  $\mathrm{I}_{\rho^2}\leq0$ yield the decay property   $ \mathrm{I}_{\rho^2}=o(\rho^2)$ as $\rho\to0^+$, since $\frac{4(p+1)-2d(p-1)}{4+d-dp}>2$ for $p>1$.
\end{proof}
\begin{lemma}
    The function $\rho\mapsto \mathrm{I}_{\rho^2}$ is continuous. 
\end{lemma}
\begin{proof}
 
Let  $\rho_n \rightarrow \rho$.
For every $n \in \mathbb{N}$, let $w_n \in S_{\rho_n}$ such that $E(w_n)\leq  \mathrm{I}_{\rho_n^2}+\frac{1}{n}$. Since we are in a mass-subcritical regime, all the terms in the energy functional are uniformly bounded.  So we easily find
\[
\mathrm{I}_{\rho^2}\leq E\left(\frac{\rho}{\rho_n}w_n\right
)= E (w_n)+o_n(1)\leq \mathrm{I}_{\rho_n^2}+o_n(1).
\]
\noindent On the other hand, given a minimizing sequence $\left\{v_n\right\} \subset S_\rho$ for $\mathrm{I}_{\rho^2}$, we have
\[
\mathrm{I}_{\rho_n^2} \leq E\left(\frac{\rho_n}{\rho} v_n\right)=E(v_n)+o_n(1)=\mathrm{I}_{\rho^2}+o_n(1),
\]
and then by combining with the previous property,  $\lim _{n \rightarrow \infty} \mathrm{I}_{\rho_n^2}=\mathrm{I}_{\rho^2}$.
\end{proof}

 We now demonstrate that if the weak non-negativity of 
$\mathrm{I}_{\rho^2}$ 
 can be strengthened to strict non-negativity, then the existence of minimizers follows. 
\begin{lemma}\label{negimplcompact}
Fix $\rho>0$ and suppose that $\mathrm{I}_{\rho^2}<0$. There exists $u \in S_{\rho}$ such that $E(u)=\mathrm{I}_{\rho^2}.$
\end{lemma}
\begin{proof}
 It suffices to show  that for any $s \in(0, \rho)$,  $\displaystyle \frac{\mathrm{I}_{s^2}}{s^2}>\frac{\mathrm{I}_{\rho^2}}{\rho^2}$. Note that the strong subadditivity at $\rho$ follows by  adding term by term  $\displaystyle\frac{\mu^2}{\rho^2}\mathrm{I}_{\rho^2}<\mathrm{I}_{\mu^2}$ and $\displaystyle\frac{(\rho^{2}-\mu^{2})}{\rho^2}\mathrm{I}_{\rho^2}<\mathrm{I}_{\rho^2-\mu^2}$. Let us define the quantity 
\[Q =\inf_{s\in (0, \rho]} \frac{\mathrm{I}_{s^2}}{s^2}.\]
From the fact that  $\mathrm{I}_{\rho^2}<0$ we have that $Q <0$, and by Proposition \ref{propasynt} that
\[
\bar\rho :=\inf\{ s \in (0, \rho] \ : \ \frac{\mathrm{I}_{s^2}}{s^2}=Q\}>0.\] 

Clearly, if 
$\bar\rho =\rho$ strong subadditivity at $
\rho$ holds, and a minimizer exists thanks to Lemma \ref{lemmainiz}. Therefore let us assume that
$\bar\rho<\rho.$
In the latter case, we have by definition that  for any $\mu \in (0, \bar\rho)$ 
\[
\frac{\mu^2}{\bar\rho^2}\mathrm{I}_{\bar\rho^2}<\mathrm{I}_{\mu^2} 
\]
and 
\[
\frac{\bar\rho^{2}-\mu^{2}}{\rho^{2}}\mathrm{I}_{\bar\rho^2}<\mathrm{I}_{\bar\rho^2-\mu^2}.
\]
Hence $
   \mathrm{I}_{\bar\rho^2} <  \mathrm{I}_{\mu^2} +\mathrm{I}_{\bar\rho^2-\mu^2}$
and 
by subadditivity, there exists $\bar{u}\in S(\bar\rho )$ with $E(\bar{u})=\mathrm{I}_{\bar\rho^2}$ and such that for $\theta \in (1-\varepsilon, 1+\varepsilon)$, for some small $\varepsilon>0$,
\[
\frac{E(\bar{u})}{\bar\rho^{2}}=\frac{\mathrm{I}_{\bar\rho^2}}{\bar\rho^{2}}\leq \frac{\mathrm{I}_{\theta^2\bar\rho^2}}{\theta^{2}\bar\rho^{2}}\leq \frac{E(\theta \bar{u})}{\theta^{2}\bar\rho^{2}}.
\]
Therefore we have 
\begin{equation}\label{eq:condderiv}
 \frac{d}{d\theta} 
\left(\theta^{2} E(\bar{u})-E(\theta \bar{u})\right)\big\lvert_{\theta=1}=0.
\end{equation}
For a minimizer $\bar{u}$ of $\mathrm{I}_{\rho^2}$ we have
\begin{equation}\label{eq:rangep}
\mathrm{I}_{\rho^2} =E(\bar{u})=\alpha\|\nabla \bar{u}\|_2^2+\beta\|\bar{u}\|_{q+1}^{q+1}-\gamma\|\bar{u}\|_{p+1}^{p+1},
\end{equation}
and 
\begin{equation}\label{eq:rangep-2}
 2\alpha\|\nabla \bar{u}\|_2^2+ \frac{d(q-1)}{2}\beta \|\bar{u}\|_{q+1}^{q+1}-\frac{d(p-1)}{2}\gamma\|\bar{u}\|_{p+1}^{p+1}=0.
\end{equation}
From \eqref{eq:condderiv} we get 
\[
\beta(q-1)\|\bar{u}\|_{q+1}^{q+1}-\gamma (p-1)\|\bar{u}\|_{p+1}^{p+1}=0,
\]
and by plugging the latter into \eqref{eq:rangep-2} we get $\|\nabla\bar{u}\|_{2}^2=0$, which is a contradiction. 
In the end, $\bar\rho=\rho$, thus subadditivity holds and a minimizer exists. 

\end{proof}

We move to the proof of the existence of a threshold mass  giving a dichotomy between $ \mathrm{I}_{\rho^2}=0$  and $ \mathrm{I}_{\rho^2}<0$. 
\begin{lemma}\label{lem:first} There exists a strictly positive threshold mass $\rho_0$ such that:\\
\textup{(i)} $ \mathrm{I}_{\rho^2}=0$  for all  $\rho\in (0,\rho_0]$;\\
\textup{(ii)} $ \mathrm{I}_{\rho^2}<0$  for all  $\rho\in (\rho_0, \infty)$.
\end{lemma}
\begin{proof}
The fact that $\mathrm{I}_{\rho^2}\leq0$, see Proposition \ref{propasynt}, together with  the weak subadditivity property implies that if $\mathrm{I}_{\rho^2}<0$, then 
$\mathrm{I}_{s^2}<0$ for all $s>\rho.$ The  negativity of $\mathrm{I}_{\rho^2}<0$ for sufficiently large $\rho$ follows by a scaling argument. Indeed, let us rescale $u$ as  $
u_{\lambda}=\lambda^{\frac{2}{q-1}}u(\lambda x)$.
We have
\[
E(u_{\lambda})=\lambda^{\frac{2(q+1)}{q-1}-d} \left( \alpha\|\nabla u\|_2^2+\beta\|u\|_{q+1}^{q+1}
\right)- \gamma\lambda^{\frac{2(p+1)}{q-1}-d}\|u\|_{p+1}^{p+1}
\]
and then $E(u_{\lambda})<0$ after choosing a sufficiently large $\lambda$, since $\frac{2(p+1)}{q-1}-d>\frac{2(q+1)}{q-1}-d$ if and only if $p>q$. On the other hand, $\|u_{\lambda}\|_2^2=\lambda^{\frac{4}{q-1}-d}\|u\|_2^2$ and hence a large $\lambda$  corresponds to a large mass $\rho$. 
 \smallskip
 
\noindent Now we prove (i), i.e., that there exists $\rho_0>0$ such that
$ \mathrm{I}_{\rho^2}=0$  for all  $\rho\in ( 0,\rho_0]$. Note that from the weak subadditivity property, together with $\mathrm{I}_{\rho^2}\leq 0$, the function $\mathrm{I}_{\rho^2}$ is non-increasing. By defining the set $O= \{\rho \ : \  \mathrm{I}_{s^2}=0 \text{ for all } s \in (0,\rho)\}$, we prove that $O\neq\emptyset$ and that 
\[
\rho_0=\sup\{\rho \ : \  \mathrm{I}_{s^2}=0 \text{ for all } s \in (0,\rho)\}>0.
\] 
Moreover, $\mathrm{I}_{\rho^2}<0$ for any $\rho>\rho_0$.
The idea is to show that $\mathrm{I}_{\rho^2}$ cannot be attained in $S_{\rho}$ if $\rho$ is sufficiently small.
As a byproduct, we will have that
\begin{equation}\label{rho-0}
\rho_0=\sup\{\rho \ : \ \mathrm{I}_{s^2}=0 \text{ for all } s \in (0,\rho]\}
\end{equation}
is strictly positive, because the negativity of $\mathrm{I}_{\rho^2}$ implies existence of minimizers, see Lemma \ref{negimplcompact}.
\smallskip

\noindent Therefore, let us assume that there exists a sequence  $\{\rho_n\}$ such that
$\rho_n \rightarrow 0^+$ and $\mathrm{I}_{\rho_n^2}$ is attained  by ground states $u_{\rho_n}$. The fact that $E(u_{\rho_n})\leq 0$ guarantees that  $\|u\|_{p+1}^{p+1}\geq \frac{\alpha}{\gamma}\|\nabla u\|_2^2$,  and jointly with the Gagliardo-Nirenberg interpolation inequality  we get
\[
\|\nabla u_{\rho_n}\|_2^2\lesssim \|u_{\rho_n}\|_{p+1}^{p+1}\lesssim \rho_n^{ p+1-\frac{d(p-1)}{2}}\|\nabla u_{\rho_n}\|_2^{\frac{d(p-1)}{2}}
\]
and so, as $1<p<1+\frac4d$, 
\begin{equation*}\label{grad-small}
\lim_{n\to+\infty}\|\nabla u_{\rho_n}\|_2=0.
\end{equation*}
On the other hand,  ground states fulfill $G(u_{\rho_n})=0$, but this contradicts the following. 
\begin{claim} For  $\varepsilon>0$ sufficiently small, the set  
\begin{equation}\label{empty}
A_\varepsilon=\{u \in H^1\setminus \{0\} \ : \  E(u)\leq 0, \ G(u)=0, \ \|\nabla u\|_{2} \leq \varepsilon \}=\emptyset.
\end{equation}
\end{claim}
\noindent The first observation  is that for a function in $A_\varepsilon$, the kinetic energy $\|\nabla u\|_2^2$ has a size comparable to that of  the focusing potential energy term  $\|u\|_{p+1}^{p+1}$. Indeed, we have
\[
0\geq E(u)-\frac{2}{d(q-1)}G(u)=\frac{d(q-1)-4}{d(q-1)}\alpha\|\nabla u\|_2^2+\frac{p-q}{q-1}\gamma\|u\|_{p+1}^{p+1}
\]
which shows that
\[
\|\nabla u\|_2^2\geq \frac{ d(p-q)}{(4-d(q-1))} \frac{\gamma}{\alpha}\|u\|_{p+1}^{p+1}.
\]
On the other hand, the non-positivity of the energy implies
\begin{equation}\label{eq:sim-p-grad}
    \|u\|_{p+1}^{p+1}\geq \frac{\alpha}{\gamma}\|\nabla u\|_2^2,
\end{equation}
and then  $\|u\|_{p+1}^{p+1}\sim \|\nabla u\|_2^2$.
Note that 
\[0\geq E(u)-\frac{1}{2}G(u)=\frac{4-d(q-1)}{4}\beta\|u\|_{q+1}^{q+1}-\frac{4-d(p-1)}{4}\gamma\|u\|_{p+1}^{p+1}.\]
The latter inequality guarantees that $
\|u\|_{p+1}^{p+1}\gtrsim \|u\|_{q+1}^{q+1}$.
Therefore, thanks to the  Gagliardo-Nirenberg inequality, we have
\[ {(\|\nabla u\|_{2}^{2})}^{\frac{d(p-q)}{2d-(d-2)(q+1)}}{(\|u\|_{q+1}^{q+1})}^{\frac{ 2(p+1)-d(p-1)}{ 2d-(d-2)(q+1)}}\gtrsim\|u\|_{p+1}^{p+1},\]
and hence
\[ {(\|\nabla u\|_{2}^{2})}^{\frac{d(p-q)}{2d-(d-2)(q+1)}}{(\|u\|_{p+1}^{p+1})}^{\frac{ 2(p+1)-d(p-1)}{ 2d-(d-2)(q+1)}}\gtrsim\|u\|_{p+1}^{p+1}.\]
We now use  \eqref{eq:sim-p-grad}, and we conclude that
\[
(\|\nabla u\|_{2}^{2})^{\frac{d(p-q)+2(p+1)-d(p-1)}{2d-(d-2)(q+1)} } \gtrsim \|\nabla u\|_{2}^{2}.
\]
We notice that
\[ \frac{d(p-q)+2(p+1)-d(p-1)}{2d-(d-2)(q+1)}  >1,
\]
when $p>q$ 
and hence $\|\nabla u\|_{2}$ cannot be too small, namely $A_\varepsilon=\emptyset$ provided that $
\varepsilon\ll1$.
\end{proof}

The next Lemma is a non-existence result which shows that if the ground state energy is zero in a open interval, then necessarily the ground state energy is not achieved. 
\begin{lemma}\label{lemnonex}
If   $\mathrm{I}_{\rho^2}=0$ in an interval $I=(0, \rho_1)$,
then for any $\rho \in I$, $\mathrm{I}_{\rho^2}$ is not achieved in $S_{\rho}$.
\end{lemma}
\begin{proof}
Let us assume the existence of a mass $\rho\in I$ and of a function $u \in S_{\rho}$ 
such that $\mathrm{I}_{\rho^2}=0=E(u)$.
Then 
\[
E(u)=\mathrm{I}_{\rho^2} \leq \mathrm{I}_{\theta^2\rho^2}\leq E(\theta u)
\]
for $\theta \in (1-\varepsilon,1+\varepsilon)$, for some small $\varepsilon>0$, and then
\[
\frac{d}{d\theta}E(\theta u)\big\lvert_{\theta=1}=0,
\]
which implies
\[
2\alpha\|\nabla u\|_{2}^{2}+(q+1)\beta\|u\|_{q+1}^{q+1}-(p+1)\gamma\|u\|_{p+1}^{p+1}=0.\]
The above condition tells us  that $u$ solves the static equation
\begin{equation*}\label{eq:stat}
- 2\alpha\Delta u  +(q+1)\beta|u|^{q-1} u-(p+1)\gamma|u|^{p-1}u=0,
\end{equation*}
and the latter is not compatible with the condition $E(u)=\mathrm{I}_{\rho^2}=0.$ To be more precise, we observe from \eqref{eq:rangep} and \eqref{eq:rangep-2} that 
\begin{equation*}\label{eq:stimenerg}
\mathrm{I}_{\rho^2} =E(u )=E(u)-\frac{2}{d(q-1)}G(u)=\alpha\left(1-\frac{4}{d(q-1)}\right)\|\nabla u \|_2^2+\frac{p-q}{q-1}\gamma\|u \|_{p+1}^{p+1}
\end{equation*}
and
\begin{equation} \label{eq:stimenerg2}
\begin{aligned}
2\alpha\|\nabla u\|_2^2&+(q+1)\beta\|u \|_{q+1}^{q+1}-(p+1)\gamma\|u \|_{p+1}^{p+1}\\
&=2\alpha\left(1-\frac{2(q+1)}{d(q-1)} \right)\|\nabla u \|_2^2+2\gamma\left(   \frac{p-q}{q-1}  \right)\|u \|_{p+1}^{p+1}.
\end{aligned}
\end{equation}
Furthermore,  due  to \eqref{eq:stimenerg2},
\[
2\alpha\|\nabla u\|_2^2+(q+1)\beta\|u\|_{q+1}^{q+1}-(p+1)\gamma\|u\|_{p+1}^{p+1}=-\frac 2d \|\nabla u\|_2^2\neq 0.
\]
This shows that  for any $\rho \in I$ minimizers for $E$ constrained on $S_{\rho}$ cannot exist.
\end{proof}
The last Lemma guarantees the existence of a ground state at the threshold mass $\rho_0$.
\begin{lemma}\label{lem-mass-cri}
For $\rho_0$ defined as in \eqref{rho-0}, there exists $u \in S_{\rho_0}$ such that $\mathrm{I}_{\rho_0^2}=E(u).$ 
\end{lemma}
\begin{proof}
Consider a  sequence $\{\rho_n\}$ of masses converging to $ \rho_0$, with $\rho_n>\rho_0$ for any $n$. By definition, $\mathrm{I}_{\rho_n^2}<0$, so denote by  $u_{\rho_n}$ a ground state that belongs to $S_{\rho_n}$. Clearly $u_{\rho_n}$ is bounded in $H^1$ and 
$\displaystyle \liminf_{n\to\infty} \|u_{\rho_n}\|_{p+1}^{p+1}>0$.
Indeed, suppose by the absurd that  $\displaystyle \lim_{k\to\infty}\|u_{\rho_{n_k}}\|_{p+1}^{p+1}=0$ along some subsequence $\{\rho_{n_k}\}$.  Then, by the negativity of the energy,  we get $\displaystyle \lim_{k\to\infty}\|\nabla u_{\rho_{n_k}}\|_2^2 =0$, and the latter is in contrast to  \eqref{empty}. 
Note that by interpolation, as $ \|\nabla u_{\rho_n}\|_{2}\lesssim 1$,
\begin{equation}\label{eq:pqr}
\|u_{\rho_n}\|_2=\rho_n\lesssim 1, \quad\|u_{\rho_n}\|_{p+1}\lesssim 1, \quad \hbox{and} \quad \|u_{\rho_n}\|_{q+1}\gtrsim 1.
\end{equation}
Define $\tilde u_{\rho_n}=\frac{\rho_0}{\rho_n} u_{\rho_n}$ which belongs to $S_{\rho_0}$.  Clearly, the estimates in \eqref{eq:pqr} are valid for $\tilde u_{\rho_n}$, and so, by the Lieb Translation Lemma \cite{Lieb}, we can claim  that up to a space translation $\tilde u_{\rho_n}\rightharpoonup \bar u\neq 0$ in $H^1$, by possibly extracting a subsequence. To prove that $\bar u \in S_{\rho_0}$,
it suffices to observe that if $\|\bar u\|_2^2=\mu^2<\rho_0^2$, then by the Brezis-Lieb Lemma \cite{BL},
\[\mathrm{I}_{\rho_0^2-\mu^2}+\mathrm{I}_{\mu^2}+o_n(1)\leq E(\tilde u_{\rho_n}-\bar u)+E(\bar u)+o_n(1)=E(\tilde u_{\rho_n})=\mathrm{I}_{\rho_0^2}+o_n(1)=o_n(1),
\]
 hence by the weak subadditivity inequality $E(\bar u)=\mathrm{I}_{\mu^2}.$ By Lemma \ref{lemnonex} this is a contradiction.
\end{proof}
 
\begin{proof}[Proof of Theorem \ref{mainthm}]
    The content of Theorem \ref{mainthm} is now a consequence  of Lemma \ref{negimplcompact}, Lemma \ref{lem:first},  Lemma \ref{lemnonex}, and Lemma \ref{lem-mass-cri}.
\end{proof}

We conclude this Section by giving the following monotonicity and continuity properties  for the threshold masses $\rho_0(\alpha,\beta, \gamma)$ with respect to the parameters.

\begin{prop}\label{monotonprop}
Let $\rho_0(\alpha, \beta, \gamma)$ be the threshold mass given by Theorem \ref{mainthm},
then 
\begin{equation*}
\begin{aligned}
\rho_0 (\alpha', \beta, \gamma)<\rho_0(\alpha, \beta, \gamma),\\
\rho_0(\alpha, \beta', \gamma)<\rho_0(\alpha, \beta, \gamma),\\
\rho_0(\alpha, \beta, \gamma' )<\rho_0(\alpha, \beta, \gamma),
\end{aligned}
\end{equation*}
for $0<\alpha'<\alpha$, $0<\beta'<\beta$, $0<\gamma<\gamma'$.
\end{prop}
\begin{proof}
The proof follows from Theorem \ref{mainthm} by noticing that for the parameters $(\alpha', \beta, \gamma)$ the energy functional  fulfills $E^{\alpha', \beta, \gamma}(u_{\rho_0})<0$, with  $u_{\rho_0}$ a ground state for the energy  $E^{\alpha, \beta, \gamma}$, namely,    $E^{\alpha, \beta, \gamma}(u_{\rho_0})=0$. The other cases are identical.
\end{proof}
\begin{prop}\label{monotonprop2}
Let $\rho_{0,E}=\rho_0\left(\frac 12 , \frac{1}{q+1}, \frac{1}{p+1}\right)$ be the threshold mass given by Theorem \ref{mainthm}. Then 
\begin{equation*}
\lim_{\eta\rightarrow 1^+}\rho_0\left(\frac 12 , \frac{1}{q+1}, \frac{\eta}{p+1}\right)=\rho_{0,E}.
\end{equation*}
\end{prop}
\begin{proof}
By the monotonicity of the threshold mass as a function of $\eta>1$ given by Proposition \ref{monotonprop},  
$\rho_0(\frac 12 , \frac{1}{q+1}, \frac{\eta}{p+1})$ has a limit as $\eta\rightarrow 1^+$. Suppose, for the sake of contradiction, that
\begin{equation*}
\lim_{\eta\rightarrow 1^+}\rho_0\left(\frac 12 , \frac{1}{q+1}, \frac{\eta}{p+1}\right)=\tilde \rho_0<\rho_{0,E}.
\end{equation*}
By denoting $\rho_0^{\eta}=\rho_0\left(\frac 12, \frac{1}{q+1}, \frac{\eta}{p+1}\right)$,  we consider a ground state $u_{\rho_0^{\eta}}$ with mass $\|u_{\rho_0^{\eta}}\|_{2}=\rho_0^{\eta}$ for the functional
$E^{\frac 12 , \frac{1}{q+1}, \frac{\eta}{p+1}}(u)$, see Theorem \ref{mainthm}.
By considering the scaling $u_{\lambda}=\lambda^{\frac{2}{q-1}}u(\lambda \cdot)$, we have
\[
E(u_{\rho_0^{\eta},\lambda})=\lambda^{\frac{2-d(q-1)+2q}{q-1}}\left(\frac12 \|\nabla u_{\rho_0^{\eta}}\|_2^2+\frac{1}{q+1} \|u_{\rho_0^{\eta}}\|_{q+1}^{q+1}-\frac{\lambda^\frac{2(p-q)}{q-1}}{p+1} \|u_{\rho_0^{\eta}}\|_{p+1}^{p+1}\right).
\]
At this point we fix $\lambda_\eta$ such that $\|u_{\rho_0^{\eta},\lambda_\eta}\|_{2}^2=\frac{\rho_{0,E}^2+\tilde \rho_0^2}{2}<\rho_{0,E}^2$. Note that  $\|u_{\rho_0^{\eta},\lambda_\eta}\|_{2}^2=\lambda_\eta^{\frac{4-d(q-1)}{q-1}}(\rho_0^{\eta})^2$, i.e., $\lambda_\eta=\left(\frac{\rho_{0,E}^2+\tilde \rho_0^2}{2(\rho_0^{\eta})^2}\right)^{\frac{q-1}{4-d(q-1)}}$.
By the fact that $\lim_{\eta\rightarrow 1^+}\rho_0^{\eta}=\tilde \rho_0$, it exists $c>1$ independent of $\eta$ such that $\lambda_\eta>c$.
We conclude by observing that $E^{\frac 12 , \frac{1}{q+1}, \frac{\eta}{p+1}}(u_{\rho_0^\eta})=0$ and then 
\[
E(u_{\rho_0^{\eta},\lambda_\eta})<\lambda_\eta^{\frac{2-d(q-1)+2q}{q-1}}\left(\frac{\eta}{p+1} \|u_{\rho_0^{\eta}}\|_{p+1}^{p+1}-\frac{c^\frac{2(p-q)}{q-1}}{p+1}\|u_{\rho_0^{\eta}}\|_{p+1}^{p+1}\right)<0
\]
when $\eta$ is sufficently close to $1$. This contradicts
that $\rho_{0,E}$ is the threshold mass.

\end{proof}

\section{Small-mass data scattering  }\label{Sec:pct}

In this Section, we collect some important facts as well as crucial results we need to prove our main achievement. First, we begin with  the following well-posedness theory.
\begin{prop}\label{prop:lwp}
  Let $d\geq 1$, $1<q<p<1+\frac4d$, and $\psi_0\in \Sigma$. Then the Cauchy problem \eqref{NLS0} is globally well-posed. In particular, the solution $\psi\in C(\mathbb R; \Sigma)$,  $M(\psi(t))=M(\psi_0)$, $E(\psi(t))=E(\psi_0)$, and $P(\psi(t))=P(\psi_0)$.
\end{prop} 

    For a proof of the well-posedness results, we refer the reader to the monograph \cite{C03}. We only recall that extension of local solution to global ones is a consequence of the blowup alternative and the fact that the leading, focusing, nonlinear term is of mass-subritical  nature.

\subsection{Pseudo-conformal transformation  and pseudo-conformal energy}

We recall  the  following pseudo-conformal transformation \[
(t,x, \psi) \mapsto (\tau,\xi,\varphi)
\]
defined by
\begin{equation}\label{eq.pct89f}
    t=\frac{\tau}{1-\tau}, \qquad x = \frac{\xi}{1-\tau},
\end{equation}   
and
\begin{equation*}\label{eq.pct89f2}
     \psi(t,x) =(1+t)^{-d/2} \varphi\left(\frac{t}{1+t},\frac{x}{1+t} \right)  e^{\frac{i|x|^2}{4(1+t)}}
    \end{equation*}
with inverse transformations given by 
\begin{equation}\label{eq.pct90f}
    \tau = \frac{t}{1+t},\qquad \xi = \frac{x}{1+t},
\end{equation}
and
  \begin{equation*}\label{eq.pct90f2}
      \varphi(\tau,\xi) =    \psi\left(\frac{\tau}{1-\tau},\frac{\xi}{1-\tau}\right) (1-\tau)^{-d/2} e^{-\frac{i  |\xi|^2}{4(1-\tau)}}.
    \end{equation*}
If $\psi$ is a solution to \eqref{NLS0} defined in the time interval $t\in [0,\infty)$, then
we can apply the pseudo-conformal transformation, and 
via
\begin{equation}
\begin{aligned}
    i \partial_t\psi +\Delta_x\psi(t,x) = e^{\frac{i|x|^2}{4(1+t)}} (1+t)^{-2-d/2}
    \left( (i \partial_\tau + \Delta_\xi)\varphi \right) \left( \frac{t}{1+t}, \frac{x}{1+t} \right) ,
\end{aligned}    
\end{equation}
we  obtain that $\varphi(\tau,\xi)$ solves 
\begin{equation}\label{eq.meq50}
\begin{aligned}
     i \partial_\tau \varphi +  \Delta_\xi \varphi &=  (1-\tau)^{-\delta(q)} |\varphi(\tau)|^{q-1}\varphi(\tau) - (1-\tau)^{-\delta(p)}|\varphi(\tau)|^{p-1}\varphi(\tau), 
    \end{aligned}
\end{equation}
on the time interval  $\tau \in [0,1)$, with initial datum
\begin{equation}\label{eqv0}
    \varphi_0(x) =  e^{-\frac{i |x|^2}{4}}\psi_0(x),
\end{equation}
where $\delta(p)$ and $\delta(q)$ are given by \eqref{def:alphas}
We give now the relations between the norms for $\varphi(\tau)$ and norms for $\psi(t)$. To this end, we define the operator 
\begin{equation}\label{def:pseudo-grad}
    J(t+1) =  \frac{x}{2}+i(1+t)\nabla_x.
\end{equation}
\begin{lemma}
The following identities hold true:
\begin{equation*}\label{eq.mrel21}
\begin{aligned}
   \|\varphi(\tau)\|_{2}&=\|\psi(t)\|_{2} , \\
       \left\|\varphi\left(\tau \right)\right\|^{r}_{r}&=(1+t)^{-d+dr/2} \left\|\psi\left(t\right)\right\|^{r}_{r}, \\
    \| |\xi| \varphi(\tau)\|_{2}&=(1+t)^{-1} \| |x| \psi(t)\|_{2},  \\
   \| \nabla_{\xi} \varphi(\tau)\|_{2}&=\|J(t+1)\psi(t)\|_{2},
\end{aligned}
\end{equation*}
where $\tau$ is defined  in \eqref{eq.pct90f}.
\end{lemma}

\begin{proof}
From \eqref{eq.pct89f}  we have
\begin{equation*}\label{eq.PCT1}
    \ |\psi(t,x)| =  \left|\varphi\left(\frac{t}{1+t},\frac{x}{1+t}\right)\right| (1+t)^{-d/2}= \left|\varphi\left(\tau,\frac{x}{1+t}\right)\right| (1+t)^{-d/2},
\end{equation*}
and then we find
\[
          \int  |\psi(t,x)|^2 dx = \int (1+t)^{-d}\left|\varphi\left(\tau,\frac{x}{1+t}\right)\right|^2 dx =\left\|\varphi\left(\tau \right)\right\|^2_{2}
\]
and
\[
         \int  |\psi(t,x)|^r dx =  \int (1+ t)^{-dr/2}\left|\varphi\left(\tau,\frac{x}{1+t}\right)\right|^r dx
         = (1+t)^{d-dr/2} \left\|\varphi\left(\tau \right)\right\|^r_{r}.
\]
In a similar way, we have
\[
          \int  |x|^2|\psi(t,x)|^2 dx = \int (1+t)^{-d+2}\left|\frac{x}{1+t} \varphi\left(\tau,\frac{x}{1+t}\right)\right|^2 dx  =(1+t)^2\left\||\xi|\varphi\left(\tau\right)\right\|^2_{2}.
\]
Finally, we have
\[
     \nabla_x \psi(t,x)  =  e^{\frac{i |\xi|^2}{4(1-\tau)}} (1-\tau)^{d/2} \left(i \frac{\xi}{2} + (1-\tau)\nabla_\xi \right)\varphi\left( \tau, \xi \right)
     \]
     and
     \[
    (1-\tau)^{d/2}\nabla_\xi \varphi(\tau,\xi)  = e^{-\frac{i |x|^2}{4(1+t)}}   \left(-i \frac{x}{2} +(1+t)\nabla_x  \right)\psi(t,x).
\]
Hence,
\begin{equation}\label{pseudo-H1}
   \int  |\nabla_\xi \varphi(\tau,\xi)|^2 d\xi= \int |\left(\frac{x}{2} +i (1+t)\nabla_x  \right)\psi(t,x)|^2 dx
   = \|J(1+t)\psi(t)\|_{2}^2. 
\end{equation}

\end{proof}

After the application of the pseudo-conformal transformation on $\psi(t)$, we consider the initial value  problem \eqref{eq.meq50}-\eqref{eqv0} for $\varphi$. We rewrite    \eqref{eq.meq50} as
\begin{equation}\label{NLS4}
    \begin{aligned}
        i \partial_\tau \varphi+\Delta_\xi\varphi = g(\tau,\varphi(\tau)),
    \end{aligned}
\end{equation}
where
\begin{equation}\label{NLS5}
    g(\tau,\varphi(\tau))= (1-\tau)^{-\delta(q)} |\varphi(\tau)|^{q-1}\varphi(\tau) - (1-\tau)^{-\delta(p)} |\varphi(\tau)|^{p-1}\varphi(\tau),
\end{equation}
and $\delta(p)$ and $\delta(q)$  defined in \eqref{def:alphas}.
These $\delta$'s will be used along the proof of the main theorem. 
Hence, the mild solution
$ \varphi \in C([0,1);\Sigma) $ to the Cauchy problem \eqref{eq.meq50} satisfies the integral equation
\begin{equation}\label{NLS8}
    \varphi(\tau) = U(\tau) \varphi_0 -i \int_0^\tau U(\tau-\sigma)g(\sigma,\varphi(\sigma)) d \sigma.
\end{equation}
We have the following equivalence result, for which we refer to \cite{C03} for a proof.
\begin{prop}\label{GWP-equi}
Let $\psi_0\in\Sigma$. 
$\psi(t)$ is a solution  to  \eqref{NLS0} in  $C([0,\infty); \Sigma)$ if and only if   
$\varphi(t)$ is a solution  to  \eqref{NLS4}-\eqref{NLS5} in  $C([0,1); \Sigma)$ with initial datum given by \eqref{eqv0}.
\end{prop}
\noindent Moreover, the following equivalence about the  asymptotic dynamics  holds true, see \cite{C03}.
\begin{prop}\label{sca-equi}
Under the hypothesis of Proposition \ref{GWP-equi}, $U(-t)\psi(t)$ has a strong limit in the $L^2$-topology as $t\to+\infty$ if and only if $\varphi(\tau)$ has a strong limit in the $L^2$-topology as $\tau\to1^-$. In particular,
\[
\lim_{t\to+\infty}U(-t)\psi(t)=\lim_{\tau\to1^-}e^{\frac{i|x|^2}{4}}U(-\tau)\varphi(\tau) \quad \hbox{ in } \quad L^2.
\]
\end{prop}
At this point, we fix a non-negative real number $A$ (that will be properly chosen later on in the paper), and we introduce the modified energy 
\begin{equation}\label{NLS9}
    \begin{aligned}
        E_A(\tau,\varphi(\tau)) &= \frac{(1-\tau)^A}{2}\|\nabla_\xi \varphi(\tau)\|_{2}^2
        +\frac{(1-\tau)^{A-\delta(q)}}{q+1} \|\varphi(\tau)\|_{q+1}^{q+1} \\
       & - \frac{(1-\tau)^{A-\delta(p)}}{p+1} \|\varphi(\tau)\|_{p+1}^{p+1}.
    \end{aligned}
\end{equation}
 It is implicit in our notation that when computing norms for the unknown $\varphi$, we are seeing them as functions of the pseudo-conformal space variable $\xi$.
We give the following.
\begin{lemma}\label{lem:A-energy}
    Let $A \geq 0$ and 
    $\varphi \in C([0,1);\Sigma)$ a solution to  \eqref{NLS8}.
Then, we have the relation
\begin{equation}\label{NLS8a}
    E_A(\tau, \varphi(\tau)) + \int_0^\tau R_A(\sigma, \varphi(\sigma)) =  E_A(0, \varphi(0)) ,
\end{equation}
where the correction energy term $R_A$ is defined by 
\begin{equation}\label{NLS10}
    \begin{aligned}
        R_A(\tau,\varphi(\tau)) &=   \frac{A(1-\tau)^{A-1}}{2}\|\nabla_\xi \varphi(\tau)\|_{2}^2\\
        &+(1-\tau)^{A-\delta(q)-1}  \frac{ (A-\delta(q))}{q+1} \|\varphi(\tau)\|_{q+1}^{q+1}\\
  &- (1-\tau)^{A-\delta(p)-1} \frac{(A-\delta(p))}{p+1} \|\varphi(\tau)\|_{p+1}^{p+1}.
    \end{aligned}
\end{equation}
\end{lemma}

\begin{proof}
After a regularization argument, we can justify the following computations. Take the real part of the $L^2$-pairing of the equation  \eqref{NLS4}
with  $(1-\tau)^A\partial_\tau \varphi$; we find
\[
\begin{aligned}
    0&=\Rea \int(i \partial_\tau \varphi+\Delta_\xi\varphi(1-\tau)^A\partial_\tau \bar\varphi\, d\xi\\
    &-\int (1-\tau)^{-\delta(q)} |\varphi(\tau)|^{q-1}\varphi(\tau) + (1-\tau)^{-\delta(p)} |\varphi(\tau)|^{p-1}\varphi(\tau))(1-\tau)^A\partial_\tau \bar\varphi\, d\xi\\
    &=\frac12(1-\tau)^A\partial_\tau\left(\Rea\int i|\varphi|^2-|\nabla\varphi|^2d\xi\right)-\frac{1}{q+1}(1-\tau)^{A-\delta(q)}\partial_\tau\left(\int|\varphi|^{q+1}d\xi\right)\\
    &+\frac{1}{p+1}(1-\tau)^{A-\delta(p)}\partial_\tau\left(\int|\varphi|^{p+1}d\xi\right).
\end{aligned}
\]
Then we have 
\begin{equation*}\label{eq.meq7a}
\begin{aligned}
  &  \frac{d}{d\tau}  \left( \frac{(1-\tau)^A}{2}\|\nabla_\xi \varphi(\tau)\|_{2}^2 +\frac{(1-\tau)^{A-\delta(q)}}{q+1} \|\varphi(\tau)\|_{q+1}^{q+1}\right)\\
  & - \frac{d}{d\tau} \left(\frac{(1-\tau)^{A-\delta(p)}}{p+1} \|\varphi(\tau)\|_{p+1}^{p+1}\right) +R_A(\tau,\varphi(\tau))   = 0, 
\end{aligned}
\end{equation*}
where $R_A(\tau,\varphi(\tau)) $ is defined in \eqref{NLS10}.
So we have
\begin{equation*}\label{NLS16}
    \frac{d}{d\tau} E_A(\tau,\varphi(\tau)) + R_A(\tau,\varphi(\tau))=0,
\end{equation*}
where $E_A$ is defined in \eqref{NLS9}. After integrating in $\tau$, we find \eqref{NLS8a} and the proof is complete.
\end{proof}
The next proposition shows that the right-hand side of \eqref{NLS8a} is non-negative if $\varphi_0$ has mass less than or equal to that of the threshold mass $\rho_{0,E}$.
\begin{lemma}\label{nonnegrhs}
The functional
\[
 E_A(0, \varphi_0):=\frac 12 \| \nabla_{\xi} \varphi_0\|_{2}^2+\frac{1}{q+1}\|\varphi_0\|_{q+1}^{q+1}-\frac{1}{p+1}\|\varphi_0\|_{p+1}^{p+1}
\]
fulfills $E_A(\varphi_0)\geq 0$ for all $\varphi_0=e^{-\frac{i |x|^2}{4}}\psi_0(x)$   in  $\Sigma$   with  mass  $\|\varphi_0\|_{2}^2=\rho^2$ if and only if $\rho\leq \rho_{0,E}$ where $\rho_{0,E}$ is the threshold mass for the existence of ground states given in Theorem \ref{mainthmforE}.
\end{lemma}
\begin{proof}
Thanks to Theorem \ref{mainthmforE}, if $\rho\leq \rho_{0,E}$ clearly $E_A(0, \varphi_0)\geq 0.$ For the  other implication  we  look at states $\varphi_0=e^{-\frac{i |x|^2}{4}}\psi_0(x).$
Let us fix $\rho>\rho_{0,E}$ and let us consider the initial condition  $\psi_0(x) =  e^{\frac{i |x|^2}{4}}u_{\rho}(x)$ where $u_{\rho}$ is a ground state (that belongs to $\Sigma$, see Appendix \ref{app}) with mass $\rho>\rho_{0,E}.$ Hence, $\varphi_0(x)=u_{\rho}(x)$ and  we have $E(u_{\rho})<0$ by Theorem \ref{mainthmforE}.
We claim that 
\begin{equation}\label{ide-sw}    
 E_A(0, \varphi_0)=E(u_\rho)<0.
\end{equation} 
By \eqref{pseudo-H1}
\[
       E_A(0, \varphi_0) = \frac{1}{2}\|J(1) \psi_0 \|_2^2+
        \frac{1}{q+1} \|u_{\rho}\|_{q+1}^{q+1} 
           -\frac{1}{p+1} \|u_{\rho}\|_{p+1}^{p+1},
\]
where $J(1)$ is defined in \eqref{def:pseudo-grad}. We compute $\|J(1) \psi_0 \|_2^2= \|(\frac{x}{2}+i\nabla_x)e^{\frac{i |x|^2}{4}}u_{\rho}\|_{2}^2$.
We have 
\begin{equation*}
    \begin{aligned}
        \|J(1) \psi_0 \|_2^2 &=\frac 14\|xu_{\rho}\|_{2}^2 +\|\nabla (e^{\frac{i |x|^2}{4}}u_{\rho})\|_2^2+  \Rea\langle xe^{\frac{i |x|^2}{4}}u_{\rho}, i \nabla (e^{\frac{i |x|^2}{4}}u_{\rho})\rangle\\
        &=\frac 14\|xu_{\rho}\|_{2}^2 +\|\nabla (e^{\frac{i |x|^2}{4}}u_{\rho})\|_2^2+  \Ima \langle xe^{\frac{i |x|^2}{4}}u_{\rho},  \nabla (e^{\frac{i |x|^2}{4}}u_{\rho})\rangle.
         \end{aligned}
\end{equation*}
Now,
\[
\|\nabla (e^{\frac{i |x|^2}{4}}u_{\rho})\|_2^2=\|\nabla u_{\rho}\|_2^2+\frac 14\|xu_{\rho}\|_{2}^2+\Rea\langle i x u_{\rho}, \nabla u_{\rho}\rangle =\|\nabla u_{\rho}\|_2^2+\frac 14\|xu_{\rho}\|_{2}^2
\]
 and
\begin{equation*}
    \begin{aligned}
 \Ima \langle xe^{\frac{i |x|^2}{4}}u_{\rho},  \nabla (e^{\frac{i |x|^2}{4}}u_{\rho})\rangle& =  \Ima \langle  x e^{\frac{i |x|^2}{4}}u_{\rho},  e^{\frac{i |x|^2}{4}}\nabla u_{\rho}+i \frac{x}{2}e^{\frac{i |x|^2}{4}}u_{\rho}\rangle\\
     &=- \Rea \langle xe^{\frac{i |x|^2}{4}}u_{\rho},  \frac{x}{2}e^{\frac{i |x|^2}{4}}u_{\rho}\rangle=-\frac 12\|xu_{\rho}\|_{2}^2.
     \end{aligned}
\end{equation*}
Hence, $ \|J(1) \psi_0 \|_2^2 =\|\nabla u_{\rho}\|_2^2$
and  hence the claim \eqref{ide-sw} holds. The proof of is concluded since $E_A(0, \varphi_0)<0$.
\end{proof}

\section{\texorpdfstring{Non-negativity of the correction energy  \eqref{NLS10} and estimates on the growth of $\|\nabla_\xi \varphi(\tau)\|_{2}^2$, $\|\varphi(\tau)\|_{q+1}^{q+1}$, $\|\varphi(\tau)\|_{p+1}^{p+1}$}{Non-negativity of the remainder}}

Aim of this subsection is to prove that we can select the exponent $A$ as introduced in Lemma \ref{lem:A-energy} satisfying  $\delta(q)<A<1$, for $\delta(q)$  defined in \eqref{def:alphas}, such that we can control of the growth of the modified energy given by \eqref{NLS9}, provided that  the initial datum has sufficiently small mass (depending on $A$).

\begin{lemma}\label{remainderestimate}
Let $d\geq 1$, $1+\frac{2}{d}< q < p < 1+ \frac{4}{d}$, and $\rho_{0,E}$ be the threshold mass for the existence of ground states given in Theorem \ref{mainthmforE}. We have the following: for any  $A\in(\delta(q),1)$, there exists a  positive mass $\rho_1(A)$, 
with 
\begin{equation}\label{mass-rho1-A}
\rho_1(A)=\rho_0\left(\frac{A}{2},\frac{(A-\delta(q))(A-\delta(p))^{-\frac{\delta(q)}{\delta(p)}}}{q+1}, \frac{1}{p+1}\right),
\end{equation}
such that for any $\varphi_0$ satisfying  $\|\varphi_0\|_{2}\leq \rho_1(A)$
\begin{equation}\label{energy-control}
    \begin{aligned}
        E_A(\tau,\varphi(\tau)) &= \frac{(1-\tau)^A}{2}\|\nabla_\xi \varphi(\tau)\|_{2}^2
        +\frac{(1-\tau)^{A-\delta(q)}}{q+1} \|\varphi(\tau)\|_{q+1}^{q+1} \\
       & - \frac{(1-\tau)^{A-\delta(p)}}{p+1} \|\varphi(\tau)\|_{p+1}^{p+1}\\
       &\leq E_A(0, \varphi(0)),
    \end{aligned}
\end{equation}
where $E_A$ is defined in \eqref{NLS9}.
\end{lemma}
\begin{proof}
Consider the correction energy functional \eqref{NLS10} in the form  
\begin{equation*}\label{NLS10bis}
    \begin{aligned}
    \tilde R_A(\tau,\varphi(\tau))&:=  (1-\tau)^{1-A}R_A(\tau,\varphi(\tau)) \\
&=\frac{A}{2}\|\nabla_\xi \varphi(\tau)\|_{L^2}^2 +(1-\tau)^{-\delta(q)}  \frac{(A-\delta(q))}{q+1} \|\varphi(\tau)\|_{q+1}^{q+1}\\
  &-  (1-\tau)^{-\delta(p)}\frac{(A-\delta(p))}{p+1} \|\varphi(\tau)\|_{p+1}^{p+1}.
    \end{aligned}
\end{equation*} 
 Denote $\varphi(\tau,\xi) = \lambda^{d/2}v(\tau, \lambda \xi)$. We get  
\begin{equation}\label{NLS10ter}
    \begin{aligned}
      \tilde R_A(\tau,\varphi(\tau))&=  \lambda^2\left( \frac{A}{2}\|\nabla_\xi v(\tau)\|_{2}^2 + (1-\tau)^{-\delta(q)} \frac{(A-\delta(q))\lambda^{-\delta(q)}}{q+1} 
      \|v(\tau)\|_{q+1}^{q+1}\right.\\
  &\left.\qquad\quad- (A-\delta(p)) \frac{(1-\tau)^{-\delta(p)}\lambda^{-\delta(p)}}{p+1} \|v(\tau)\|_{p+1}^{p+1}\right).
    \end{aligned}
\end{equation}
Fix  $\lambda$ so that  $ (A-\delta(p))(1-\tau)^{-\delta(p)}\lambda^{-\delta(p)}=1$, namely $\lambda=(A-\delta(p))^{1/\delta(p)}(1-\tau)^{-1}$. By plugging this choice of the parameter into \eqref{NLS10ter}, we obtain
\begin{equation*}\label{NLS10quater}
    \begin{aligned}
      \tilde R_A(\tau,\varphi(\tau))&=  \frac{(A-\delta(p))^{2/\delta(p)}}{(1-\tau)^2}\left(\frac{A}{2}\|\nabla_\xi v(\tau)\|_{2}^2\right. \\
        & \left.+ \frac{(A-\delta(q))(A-\delta(p))^{-\frac{\delta(q)}{\delta(p)}}}{q+1} \|v(\tau)\|_{q+1}^{q+1}- \frac{1}{p+1} \|v(\tau)\|_{p+1}^{p+1}\right).
    \end{aligned}
\end{equation*}
At this point, we look at the functional
\begin{equation}\label{stimremimp}
    \begin{aligned}
      \tilde R_A'(\tau,\varphi(\tau))&:=\frac{(1-\tau)^{2}}{(A-\delta(p))^{2/\delta(p)}}\tilde R_A(\tau,\varphi(\tau))\\
      &=  \left( \frac{A}{2}\|\nabla_\xi v(\tau)\|_{2}^2+ \frac{(A-\delta(q))(A-\delta(p))^{-\frac{\delta(q)}{\delta(p)}}}{q+1} \|v(\tau)\|_{q+1}^{q+1}\right.\\&\left.\qquad\quad- \frac{1}{p+1} \|v(\tau)\|_{p+1}^{p+1}\right).
    \end{aligned}
\end{equation}
Note that the  right-hand side of \eqref{stimremimp} in non-negative provided that 
\[
\|\varphi_0\|_{2}\leq \rho_1(A)=\rho_0\left(\frac{A}{2},\frac{(A-\delta(q))(A-\delta(p))^{-\frac{\delta(q)}{\delta(p)}}}{q+1},\frac{1}{p+1}\right)
\]
where $\rho_0$ is defined in \eqref{threshold-mass} of Theorem \ref{mainthm}.
Hence, the positivity of the correction energy term implies \eqref{energy-control}.
\end{proof}
The previous Lemma shows  that we are allowed to take an arbitrary $A\in (\delta(q),1)$ in order to prove that the correction term given by \eqref{NLS10} is non-negative when the initial datum has $L^2$-norm smaller than $\rho_1(A)$. Now we prove that this threshold mass is monotone increasing with respect to the parameter $A$.
\begin{lemma}\label{mono} Let $\rho_1(A)$ as defined in \eqref{mass-rho1-A}. We have that 
\begin{equation}\label{eq:sup-rhoA}
   \sup_{A\in (\delta(q),1)}\rho_1(A) = \rho_0\left(\frac{1}{2},\frac{(1-\delta(q))(1-\delta(p))^{-\frac{\delta(q)}{\delta(p)}}}{q+1}, \frac{1}{p+1}\right)
    \end{equation}
and 
\begin{equation}\label{eq:sup-rho-Ener}
\rho_0\left(\frac{1}{2},\frac{(1-\delta(q))(1-\delta(p))^{-\frac{\delta(q)}{\delta(p)}}}{q+1}, \frac{1}{p+1}\right)<\rho_{0,E}.
\end{equation}
Therefore, $E_A(0, \varphi(0))>0$ for all $A\in (\delta(q),1)$.
\end{lemma}
\begin{proof}
Set $h(A)=(A-\delta(q))(A-\delta(p))^{-\frac{\delta(q)}{\delta(p)}}$ and consider 
the functional 
\[
E^{\frac{A}{2},\frac{h(A)}{q+1},\frac{1}{p+1}}(\varphi)=\left( \frac{A}{2}\|\nabla_\xi \varphi(\tau)\|_{2}^2+ \frac{h(A)}{q+1} \|\varphi(\tau)\|_{q+1}^{q+1}- \frac{1}{p+1} \|\varphi(\tau)\|_{p+1}^{p+1}\right).
\]
 Define $\varphi(\tau,\xi)=\lambda^{d/2}v(\tau, \lambda \xi)$. We get
\[
\frac{E^{\frac{A}{2},\frac{h(A)}{q+1},\frac{1}{p+1}}(\tau,\varphi(\tau) )}{A}=\lambda^2 \left( \frac{1}{2}\|\nabla_\xi v(\tau)\|_{2}^2+ \frac{h(A)\lambda^{-\delta(q)}}{A(q+1)} \|v(\tau)\|_{q+1}^{q+1}- \frac{\lambda^{-\delta(p)}}{A(p+1)} \|v(\tau)\|_{p+1}^{p+1}\right)
\]
and hence
\[
\begin{aligned}
\frac{E^{\frac{A}{2},\frac{h(A)}{q+1},\frac{1}{p+1}}(\tau,\varphi(\tau) )}{A}=&\lambda^2 \left( \frac{1}{2}\|\nabla_\xi v(\tau)\|_{2}^2+ \frac{1}{q+1} \|v(\tau)\|_{q+1}^{q+1}\right.\\
&\qquad \quad -\left. \frac{1}{p+1}\frac{1}{A}\left(\frac{h(A)}{A}\right)^{-\frac{\delta(p)}{\delta(q)}} \|v(\tau)\|_{p+1}^{p+1}\right)
\end{aligned}
\]
when $\lambda=\left(A^{-1}h(A)\right)^{\frac{1}{\delta(q)}}.$ With this choice of $\lambda$, the function 
\[
f(A):=A^{-1}\left(A^{-1}h(A)\right)^{-\frac{\delta(p)}{\delta(q)}}=\left(1-\frac{\delta(p)}{A} \right)\left(1-\frac{\delta(q
)}{A}\right)^{-\frac{\delta(p)}{\delta(q)}}.
\]
For readability convenience, we set $x=\delta(q)$ and $y=\delta(p)$. Recall that $y<x<A<1$. By computing the derivative of $f(A)$ we get 
\[
\begin{aligned}
f'(A)&=yA^{-2} \left(1-\frac{x}{A} \right)^{-\frac{y}{x}}-y A^{-2}\left(1-\frac{y}{A} \right)\left(1-\frac{x}{A} \right)^{-\frac{y}{x}-1}\\
&=yA^{-3}\left(1-\frac{x}{A} \right)^{-\frac{y}{x}-1}(y-x)<0
\end{aligned}
\]
and then $f(A)$ is decreasing, so 
\eqref{eq:sup-rhoA} is proved. To get \eqref{eq:sup-rho-Ener}, it suffices to note that the function $(1-\delta(q))(1-\delta(p))^{-\frac{\delta(q)}{\delta(p)}}<1$. We omit the computations. 
\end{proof}

 We now exhibit an explicit upper bound on the growth of $\|\nabla_\xi \varphi(\tau)\|_{2}^2$, $\|\varphi(\tau)\|_{q+1}^{q+1}$, and $\|\varphi(\tau)\|_{p+1}^{p+1}$, by means of negative powers of $(1-\tau)$.
\begin{lemma}
 Let $d\geq1$, $1+\frac{2}{d}< q < p < 1+ \frac{4}{d}$, and  $\rho_{0,E}>0$ be the threshold mass for the existence of ground states as in Theorem \ref{mainthmforE}. Moreover, let $A\in(\delta(q),1)$ and $\rho_1(A)$ be as given in Lemma \ref{remainderestimate}. Then, 
 for $\|\varphi_0\|_{2}\leq \rho_1(A)$,
\begin{equation}\label{control-tau}
    \begin{aligned}
        \|\nabla_\xi \varphi(\tau)\|_{2}^2&\lesssim (1-\tau)^{-A},\\
\|\varphi(\tau)\|_{q+1}^{q+1}&\lesssim (1-\tau)^{\delta(q)-A},\\
\|\varphi(\tau)\|_{p+1}^{p+1}&\lesssim (1-\tau)^{\delta(p)-A}.
    \end{aligned}
\end{equation}
\end{lemma}
\begin{proof}
Fix $0<\varepsilon<1$. We write $ E_A(\tau,\varphi(\tau))$ as
\begin{equation*}
 \begin{aligned}
        E_A(\tau,\varphi(\tau)) &= \frac{\varepsilon(1-\tau)^A}{2}\|\nabla_\xi \varphi(\tau)\|_{2}^2 +\frac{\varepsilon(1-\tau)^{A-\delta(q)}}{q+1} \|\varphi(\tau)\|_{q+1}^{q+1} \\
       & +\frac{\varepsilon(1-\tau)^{A-\delta(p)}}{p+1} \|\varphi(\tau)\|_{p+1}^{p+1}+E_A^{\star}(\tau,\varphi(\tau)), 
    \end{aligned}
\end{equation*}
where
\begin{equation*}
 \begin{aligned}
        E_A^{\star}(\tau,\varphi(\tau)) &= \frac{(1-\varepsilon)(1-\tau)^A}{2}\|\nabla_\xi \varphi(\tau)\|_{2}^2\\
       & +\frac{(1-\varepsilon)(1-\tau)^{A-\delta(q)}}{q+1} \|\varphi(\tau)\|_{q+1}^{q+1} \\
&        -\frac{(1+\varepsilon)(1-\tau)^{A-\delta(p)}}{p+1} \|\varphi(\tau)\|_{p+1}^{p+1}.
    \end{aligned}
\end{equation*}
Our purpose is to show that for a suitable choice of $\varepsilon$ with $\varepsilon>\varepsilon_0$, for $\varepsilon_0$ independent of $A$,  $E_A^{\star}(\tau,\varphi(\tau))$ is strictly positive for all $\tau\in (0,1)$. We argue as before by considering
\begin{equation}\label{NLS10bis2}
    \begin{aligned}
      \tilde E_A^{\star}(\tau,\varphi(\tau))&:=  (1-\tau)^{-A}E_A^{\star}(\tau,\varphi(\tau)) \\
      &=\frac{1-\varepsilon}{2}\|\nabla_\xi \varphi(\tau)\|_{2}^2\\
        & + (1-\varepsilon) \frac{(1-\tau)^{-\delta(q)}}{q+1} \|\varphi(\tau)\|_{q+1}^{q+1}\\
  &- (1+\varepsilon) \frac{(1-\tau)^{-\delta(p)}}{p+1} \|\varphi(\tau)\|_{p+1}^{p+1}.
    \end{aligned}
\end{equation}
 Define $\varphi(\tau,\xi)=\lambda^{d/2}v(\tau,\lambda\xi)$. We have 
\begin{equation*}\label{NLS10bis3}
    \begin{aligned}
      \tilde E_A^{\star}(\tau,\varphi(\tau))&: = \lambda^2 \left(\frac{1-\varepsilon}{2}\|\nabla_\xi v(\tau)\|_{2}^2\right.\\
        & + (1-\varepsilon) \frac{(1-\tau)^{-\delta(q)}\lambda^{-\delta(q)}}{q+1} \|v(\tau)\|_{q+1}^{q+1}\\
  &- \left.(1+\varepsilon) \frac{(1-\tau)^{-\delta(p)}\lambda^{-\delta(p)}}{p+1} \|v(\tau)\|_{p+1}^{p+1}\right).
    \end{aligned}
\end{equation*}
We impose that $ (1-\tau)^{-\delta(p)}\lambda^{\frac{d}{2}(p-1)-2}=1$, namely $\lambda=(1-\tau)^{-1}$ and we get
\[
\begin{aligned}
  \tilde E_A^{\star}(\tau,\varphi(\tau))&=(1-\tau)^{-2}\left( \frac{1-\varepsilon}{2}\|\nabla_\xi v(\tau)\|_{2}^2 + \frac{(1-\varepsilon)}{q+1}  \|v(\tau)\|_{q+1}^{q+1}\right.\\
  &\left.\qquad \qquad \qquad - \frac{(1+\varepsilon)}{p+1} \|v(\tau)\|_{p+1}^{p+1}\right), 
\end{aligned}
\]
and then 
\[
\frac{(1-\tau)^{2}}{1-\varepsilon}\tilde E_A^{\star}(\tau,\varphi (\tau))=\left( \frac{1}{2}\|\nabla_\xi v(\tau)\|_{2}^2 + \frac{1}{q+1}  \|v(\tau)\|_{q+1}^{q+1}- \frac{(1+\varepsilon)}{(1-\varepsilon)(p+1)} \|v(\tau)\|_{p+1}^{p+1}\right).
\]
By Proposition \ref{monotonprop} there exists $\rho_2(\varepsilon)=\rho_0\left(\frac 12, \frac{1}{q+1}, \frac{(1+\varepsilon)}{(1-\varepsilon)(p+1)}\right)$ with $0<\rho_2(\varepsilon)<\rho_{0,E}$ such that 
$\tilde E_A^{\star}>0$ for all $\varphi$ with $\|\varphi\|_{L^{2}}\leq \rho_2(\varepsilon).$ By Proposition \ref{monotonprop2} we have 
$
\displaystyle\lim_{\varepsilon \to 0^+}\rho_2(\varepsilon)=\rho_{0,E}$,
and by the fact that 
\[
\rho_1(A)<\rho_0\left(\frac{1}{2},\frac{(1-\delta(q))(1-\delta(p))^{-\frac{\delta(q)}{\delta(p)}}}{q+1}, \frac{1}{p+1}\right)<\rho_{0,E}
\]
we can select $\varepsilon$ sufficiently small  independent of $A$ such that $\tilde E_A^{\star}(\tau,\varphi_{\lambda}(\tau))\geq 0$,
and hence  for all $\varphi_0$ with $\|\varphi_0\|_{L^{2}}\leq \rho_1(A)$
 we have \eqref{control-tau}.
\end{proof}

\begin{remark}
    Notice that if we consider the initial condition $\psi_0=u_{\rho_{0,E}}$,  with $u_{\rho_{0,E}}$  being a ground state with mass $\rho_{0,E}$, we have
    \begin{equation}\label{remei}
    \lim_{\tau \to 1^-} \tilde E_A^{\star}(\tau,\varphi(\tau))=-\infty
    \end{equation} 
    where $\tilde E_A^{\star}(\tau,\varphi(\tau))$ is given by \eqref{NLS10bis2}.
    We emphasize that this fact is crucial, because it shows that \eqref{control-tau} cannot hold for the time evolution of the ground state. To show \eqref{remei}
we just notice that, being $1-\tau=(1+t)^{-1}$,
\begin{equation}
    \begin{aligned}\tilde E_A^{\star}(\tau,\varphi(\tau))& =\frac{(1-\varepsilon)(1+t)^2}{2}\|u_{\rho_{0,E}}\|_2^2 +\frac{(1-\varepsilon)(1+t)^2}{q+1}\|u_{\rho_{0,E}}\|_{q+1}^{q+1}\\
    &-\frac{(1+\varepsilon)(1+t)^2}{p+1}\| u_{\rho_{0,E}}\|_{p+1}^{p+1}+\frac{(1-\varepsilon)(1+t)}{2}\langle x u_{\rho_{0,E}}, \nabla u_{\rho_{0,E}}\rangle\\
    &+\frac{1-\varepsilon}{8}\|xu_{\rho_{0,E}}\|_2^2
\end{aligned}
\end{equation}
and hence, by the fact that $
E(u_{\rho_{0,E}})=0$, 
\[
\begin{aligned}
\tilde E_A^{\star}(\tau,\varphi(\tau))&<-\frac{\varepsilon(1+t)^2}{2}\|\nabla u_{\rho_{0,E}} \|_2^2+\frac{(1-\varepsilon)(1+t)}{2}\langle x  u_{\rho_{0,E}}, \nabla  u_{\rho_{0,E}}\rangle\\
&+\frac{1-\varepsilon}{8}\|x  u_{\rho_{0,E}}\|_2^2\rightarrow -\infty \quad \hbox{ as } \quad   t\to+\infty.
\end{aligned}
\]
\end{remark}

At this point we can prove the main scattering result, by implementing the Tsutsumi and Yajima scheme. 

\subsection{Proof of Theorem \ref{mainthm-sca}} 
Let $A\in(\delta(q),1)$ and $\rho_1(A)$ be as given in Lemma \ref{remainderestimate}. Then \eqref{control-tau}
holds and 
\begin{equation*}
    \begin{aligned}
        \|\nabla_\xi \varphi(\tau)\|_{2}^2&\lesssim (1-\tau)^{-A},\\
\|\varphi(\tau)\|_{q+1}^{q+1}&\lesssim (1-\tau)^{\delta(q)-A},\\
\|\varphi(\tau)\|_{p+1}^{p+1}&\lesssim (1-\tau)^{\delta(p)-A}.
    \end{aligned}
\end{equation*}
By Proposition \ref{sca-equi}, we prove that $\varphi(\tau)$ as a strong limit in the $L^2$-topology as $\tau\to1^-$. Recall the embedding $L^{\frac{p+1}{p}}\subset H^{-1}\subset H^{-2}$, then by using \eqref{NLS4}-\eqref{NLS5}, \eqref{control-tau}, and the conservation of the mass for $\varphi(\tau)$, 
\[
\begin{aligned}
\|\partial_{\tau} \varphi(\tau)\|_{H^{-2}}&\leq \|\varphi(\tau)\|_{2}+C(1-\tau)^{-\delta(q)}\|\varphi(\tau)|\varphi(\tau)|^{q-1}\|_{H^{-2}}\\
&+C(1-\tau)^{-\delta(p)}\|\varphi(\tau)|\varphi(\tau)|^{p-1}\|_{H^{-2}}\\
&\leq \|\varphi(\tau)\|_{2}+C(1-\tau)^{-\delta(q)}\|\varphi(\tau)\|_{q+1}^{q}\\
&+C(1-\tau)^{-\delta(p)}\|\varphi(\tau)\|_{p+1}^{p}\\
&\leq \|\varphi_0\|_{2}+C(1-\tau)^{-\delta(q)+\frac{q}{q+1}(\delta(q)-A)}\\
&+C(1-\tau)^{-\delta(p)+\frac{p}{p+1}(\delta(p)-A)}.
\end{aligned}
\]
This means that $\partial_\tau \varphi $ belongs to $L^1((0,1);H^{-2})$  since $\delta(p)-\frac{p}{p+1}(\delta(p)-A)<1$ for any $p\in\left(1,1+\frac 4d\right)$, and this implies the existence of a function $\tilde \varphi\in H^{-2}$ such that $\varphi(\tau)\to \tilde\varphi$ in the $H^{-2}$-topology as $\tau\to1^-$. By using again the conservation of the mass for $\varphi(\tau)$, we actually have that $\tilde \varphi\in L^2$ and by the uniform bound in $L^2$, $\varphi(\tau)\rightharpoonup \tilde \varphi$, as $\tau\to1^-$. Let us consider $0\leq \tau'\leq \tau <1$. By using the fundamental theorem of calculus, and again the  equation solved by $\varphi(\tau)$, see \eqref{NLS4}-\eqref{NLS5}, for a test function $\phi$ we have, 
\[
\begin{aligned}
\langle\varphi(\tau)-\varphi(\tau'), \phi\rangle  & =\int_{\tau'}^\tau \langle i\nabla \varphi(t), \nabla\phi\rangle ds \\
&+\int_{\tau'}^\tau(1-s)^{-\delta(q)}\langle i |\varphi(s)|^{q-1}\varphi(s), \phi\rangle ds\\
&-\int_{\tau'}^\tau(1-s)^{-\delta(p)}\langle i |\varphi(s)|^{p-1}\varphi(s), \phi\rangle ds, 
\end{aligned}
\]
and then 
\[
\begin{aligned}
|\langle\varphi(\tau)-\varphi(\tau'), \phi\rangle|  & \leq \|\nabla \phi\|_{2}\int_{\tau'}^\tau \|\nabla \varphi(s)\|_{2}ds +\|\phi\|_{L^{q+1}}\int_{\tau'}^\tau(1-s)^{-\delta(q)}\|\varphi(s)\|_{q+1}^qds\\
&+\|\phi\|_{p+1}\int_{\tau'}^\tau(1-s)^{-\delta(p)}\|\varphi(s)\|_{p+1}^pds\\
&\leq \|\nabla \phi\|_{2}\int_{\tau'}^\tau (1-s)^{-A/2}ds +\|\phi\|_{q+1}\int_{\tau'}^\tau (1-s)^{-\delta(q)+\frac{q}{q+1}(\delta(q)-A)} ds\\
&+\|\phi\|_{p+1}\int_{\tau'}^\tau(1-s)^{-\delta(p)+\frac{p}{p+1}(\delta(p)-A)} ds.
\end{aligned}
\]
By the weak convergence in $L^2$, we have 
\begin{equation}\label{decay-tilde}
\begin{aligned}
|\langle\tilde \varphi-\varphi(\tau'), \phi\rangle|  
& \leq \|\nabla \phi\|_{2}\int_{\tau'}^1 (1-s)^{-A/2}ds +\|\phi\|_{q+1}\int_{\tau'}^1 (1-s)^{-\delta(q)+\frac{q}{q+1}(\delta(q)-A)} ds\\
&+\|\phi\|_{p+1}\int_{\tau'}^1(1-s)^{-\delta(p)+\frac{p}{p+1}(\delta(p)-A)} ds\\
&\lesssim (1-\tau')^{-\frac A2+1}\|\nabla \phi\|_{2}+(1-\tau')^{-\delta(q)+\frac{q}{q+1}(\delta(q)-A)+1}\|\phi\|_{q+1}\\
&+(1-\tau')^{-\delta(p)+\frac{p}{p+1}(\delta(p)-A)+1}\|\phi\|_{p+1}.
\end{aligned}
\end{equation}
At this point we set $\phi=\varphi(\tau')$ in \eqref{decay-tilde}, and we get, by using once more the controls given in \eqref{control-tau},
\[
\begin{aligned}
|\langle\tilde \varphi-\varphi(\tau'), \varphi(\tau')\rangle|&\lesssim (1-\tau')^{-A+1}+(1-\tau')^{-\delta(q)+\frac{q}{q+1}(\delta(q)-A)+1 +\frac{q}{q+1}(\delta(q)-A)}\\
&+(1-\tau')^{-\delta(p)+\frac{p}{p+1}(\delta(p)-A)+1+\frac{p}{p+1}(\delta(p)-A)},
\end{aligned}
\]
and the right hand side of the above estimate converges to zero as $\tau'\to 1^-$. Indeed, it goes to zero provided that both the exponent $-\delta(q)+\frac{q}{q+1}(\delta(q)-A)+1 +\frac{q}{q+1}(\delta(q)-A)$ and $-\delta(p)+\frac{p}{p+1}(\delta(p)-A)+1+\frac{p}{p+1}(\delta(p)-A)$ are positive. Let us consider the term in $q$, the other being similar. We have that the quantity
\[
 -\delta(q)+\frac{q}{q+1}(\delta(q)-A)+1 +\frac{q}{q+1}(\delta(q)-A)=-\delta(q)+\frac{2}{q+1}(\delta(q)-A)+1
\]
is strictly positive if and only if 
\begin{equation}\label{cond:Aqp}
    A<\frac{1-q}{2}\left(2-\frac d2(q-1)\right)+\frac{q+1}{2}.
\end{equation}
Note that the right-hand side of \eqref{cond:Aqp} above is larger  than 1 for $q>1+\frac2d$. Hence, since $A<1$, \eqref{cond:Aqp} is verified. 
The proof is concluded by noting that 
\[
\|\varphi(\tau)-\tilde \varphi\|_{2}^2=-\langle\tilde \varphi-\varphi(\tau),\varphi(\tau)\rangle+\langle\tilde \varphi-\varphi(\tau), \tilde \varphi\rangle,
\]
which tends to zero by the previous convergence properties. \\
In conclusion, we  proved $L^2$ scattering for any initial datum with mass smaller then $\rho_1(A)$, the latter being defined in \eqref{mass-rho1-A}. Now we pass to the limit when 
$A\rightarrow 1^-$, and by using the monotonicity of the threshold mass given by Lemma \ref{mono}, we get scattering below the mass given by  \eqref{eq:sup-rhoA}.

\appendix
\section{}\label{app}
\begin{prop}
Let $u\in H^1 $ be a standing  wave solution to
\begin{equation*}
- \Delta u + \omega u +|u|^{q-1}u-|u|^{p-1}u=0,
\end{equation*}
with $E(u)\leq 0$. Then $\omega>0$ and $u \in \Sigma.$
\end{prop}
\begin{proof}
Let us define $K=\|\nabla u\|_{2}^2$, $ N_q= \|u\|^{q+1}_{q+1}$, $ N_p=\|u\|^{p+1}_{p+1}$, and $ M=\|u\|^{2}_{2}$. Hence the standing wave solve the following system of equations
\begin{equation}\label{systemstand}
\begin{cases}
   K+N_q-N_p+\omega M =0 \\
 \frac 12 K+\frac{1}{q+1}N_q-\frac{1}{p+1}N_p =E\\
 K+\frac{d}{2}\frac{(q-1)}{q+1}N_q-\frac{d}{2}\frac{(p-1)}{p+1}N_p=0.
\end{cases}
\end{equation}
From \eqref{systemstand} we derive
\begin{equation}\label{systemstand2}
\begin{cases}
  N_p=\frac{2(p+1)}{d(p-1)}K+\frac{(q-1)(p+1)}{(q+1)(p-1)}N_q \\
\frac{d(p-1)-4}{2d(p-1)}K+\frac{p-q}{(q+1)(p-1)}N_q=E
\end{cases},
\end{equation}
while from \eqref{systemstand2} we get 
\begin{equation}\label{systemstand3}
 N_q=\frac{(q+1)(p-1)}{p-q}E-\frac{(d(p-1)-4)(q+1)}{2d(p-q)}K 
\end{equation}
and hence we have
\begin{equation*}\label{systemstand4}
   K+N_q -\left(\frac{2(p+1)}{d(p-1)}K+\frac{(q-1)(p+1)}{(q+1)(p-1)}N_q\right)+\omega M=0,
\end{equation*}
which implies
\begin{equation*}\label{systemstand5}
   \left(\frac{d(p-1)-2p-2}{d(p-1)}\right)K+N_q\left(\frac{2(p-q)}{(q+1)(p-1)}\right)+\omega M=0.
\end{equation*}
From \eqref{systemstand3}
\begin{equation*}\label{systemstand6}
   -\frac{2}{d}K+2E+\omega M=0,
\end{equation*}
which implies $\omega>0$ if $E(u)\leq 0.$ If we consider the ground state, once the positivity of $\omega$ is shown,
the exponential decay is well-known, see \cite{BL83},
and hence the ground state belongs to $\Sigma$.
\end{proof}
We claim the following homogeneity properties of the threshold masses.
\begin{lemma}
    The function $\rho_0\left(\alpha,\beta,\gamma\right)$ is homogeneous of order $0$ and satisfies
    \begin{equation}\label{ide-abc}
\rho_0\left(\alpha,\beta,\gamma\right)=\rho_0\left(1,1, \Lambda(\alpha,\beta,\gamma)\right)
\end{equation}
where
\begin{equation}
   \Lambda(\alpha,\beta,\gamma) = \frac{\gamma}{\alpha^{1-\frac{\delta(p)}{\delta(q)}} \beta^{\frac{\delta(p)}{\delta(q)}}}. 
\end{equation}
\end{lemma}
\begin{proof}
   Let us consider  $E^{\alpha,\beta,\gamma}(u)$ defined in \eqref{ener-general}. Then, for $u_{\lambda}=\lambda^{d/2}u(\lambda x)$ and $\lambda=\left(\frac{\beta}{\alpha}\right)^{\frac{1}{\delta(q)}}$, we get 
\[
E^{\alpha,\beta,\gamma}(u_{\lambda})=\alpha\lambda^2\left(\|\nabla u\|_{2}^2+\| u\|_{q+1}^{q+1}-\frac{\gamma}{\alpha^{1-\frac{\delta(p)}{\delta(q)}} \beta^{\frac{\delta(p)}{\delta(q)}}}\| u\|_{p+1}^{p+1}\right).
\]
The above identity implies \eqref{ide-abc} by definition. 
\end{proof}
As a by-product we can give the order of the threshold masses. 
\begin{prop}
The masses $\rho^{\star}$, $\rho_{0,SW}$, and $\rho_{0,E}$ defined in \eqref{defenergsw}, in Theorem \ref{mainthm-sca},  in Theorem \ref{mainthmforE}, respectively, are ordered as follows:
\[\rho^{\star}<\rho_{0, SW}<\rho_{0,E}.\]
\end{prop}
\begin{proof}
Note that 
\[
     \rho^{\star}=\rho_0\left(\frac{1}{2}, \frac{(1-\delta(q))(1-\delta(p))^{-\frac{\delta(q)}{\delta(p)}}}{q+1}, \frac{1}{p+1}\right)
    =\rho_0\left(\alpha_1, \beta_1, \gamma_1\right) 
    \]
 with $(\alpha_1, \beta_1, \gamma_1)=\left(\frac12, \frac{(1-\delta(q))(1-\delta(p))^{-\frac{\delta(q)}{\delta(p)}}}{q+1},\frac{1}{p+1}\right)$,
\[
 \rho_{0, SW}=\rho_0\left(1, \frac{d(q-1)}{2(q+1)}, \frac{d(p-1)}{2(p+1)}\right) =\rho_0\left(\alpha_2, \beta_2, \gamma_2\right)
 \]
 with $(\alpha_2, \beta_2,\gamma_2)=\left(1, \frac{d(q-1)}{2(q+1)},\frac{d(p-1)}{2(p+1)}\right)$, and  
\[ 
 \rho_{0, E}=\rho_0\left(\frac 12, \frac{1}{q+1}, \frac{1}{p+1}\right)=\rho_0\left(\alpha_3,\beta_3,\gamma_3\right) 
 \]
with $(\alpha_3, \beta_3,\gamma_3)=\left(\frac12, \frac{1}{q+1},\frac{1}{p+1}\right)$. By  \eqref{ide-abc} and the monotonicity property of Proposition \ref{monotonprop}, we aim at proving that $\Lambda(\alpha_1,\beta_1,\gamma_1)>\Lambda(\alpha_2,\beta_2,\gamma_2)>\Lambda(\alpha_3,\beta_3,\gamma_3)$. 
The latter is equivalent to prove that 
\begin{equation}\label{tri-ineq}
    (1-\delta(q))^{-\frac{\delta(p)}{\delta(q)}}(1-\delta(p))> \frac{d(p-1)}{4}\left(\frac{d(q-1)}{4}\right)^{-\frac{\delta(p)}{\delta(q)}}>1.
\end{equation}
By definition, $\delta(p)=2-\frac{d(p-1)}{2}$ and $\delta(q)=2-\frac{d(q-1)}{2}$, thus \eqref{tri-ineq} becomes
\begin{equation}\label{tri-ineq-bis}
    (1-\delta(q))^{-\frac{\delta(p)}{\delta(q)}}(1-\delta(p))>  \left(1-\frac{\delta(q)}{2}\right)^{-\frac{\delta(p)}{\delta(q)}}\left(1-\frac{\delta(p)}{2}\right)>1.
\end{equation}
We introduce the function
\[
F(x)=\left(1-\frac{\delta(q)}{x}\right)^{-\frac{\delta(p)}{\delta(q)}}\left(1-\frac{\delta(p)}{x}\right), \qquad x \in [1, \infty).
\]
Note that \eqref{tri-ineq-bis} is equivalent to 
\[
F(1)>F(2)>\lim_{x\to +\infty}F(x),
\]
and the proof is concluded provided that $F$   is monotone decreasing. Direct computations give
\[
F'(x)=\frac{\delta(p)}{x^3}\left(1-\frac{\delta(q)}{x}\right)^{-\frac{\delta(p)}{\delta(q)}-1}(\delta(p)-\delta(q))<0,
\]
by recalling that  $\delta(p)<\delta(q)$. The proof is concluded. 
\end{proof}

\subsection*{Acknowledgements}

JB, LF, and VG were supported in part by the GNAMPA
  (Gruppo Nazionale per l'Analisi Matematica) of the INdAM.  VG was supported in part by the Institute of Mathematics and Informatics, Bulgarian Academy of Sciences and by Top Global University Project, Waseda University.

\bibliographystyle{plain}
\bibliography{scattering_mass_sub}

@article{Barab,
 author = {Barab, J. E.},
 title = {Nonexistence of asymptotically free solutions for a nonlinear {Schr{\"o}dinger} equation},
 fjournal = {Journal of Mathematical Physics},
 journal = {J. Math. Phys.},
 issn = {0022-2488},
 volume = {25},
 pages = {3270--3273},
 year = {1984},
 doi = {10.1063/1.526074},
 zbMATH = {3882915},
 Zbl = {0554.35123}
}

@article{BF2025,
author = { Bellazzini, J. and  Forcella, L.},
title = {Mass-subcritical {H}alf-{W}ave equation with mixed nonlinearities: existence and non-existence of ground states and traveling waves},
journal = {Discrete and Continuous Dynamical Systems, to appear},
year = {2025},
doi = {10.3934/dcds.2025131},
url = {https://www.aimsciences.org/article/id/689da9459b25f458eb055cd6}
}

@article{BGTV,
 author = {Burq, N. and Georgiev, V. and Tzvetkov, N. and Visciglia, N.},
 title = {${H^1}$ scattering for mass-subcritical {NLS} with short-range nonlinearity and initial data in ${\Sigma}$},
 fjournal = {Annales Henri Poincar{\'e}},
 journal = {Ann. Henri Poincar{\'e}},
 issn = {1424-0637},
 volume = {24},
 number = {4},
 pages = {1355--1376},
 year = {2023},
 doi = {10.1007/s00023-022-01245-2},
 zbMATH = {7678834},
 Zbl = {1514.35404}
}

@book {C03,
    AUTHOR = {Cazenave, T.},
     TITLE = {Semilinear {S}chr\"{o}dinger equations},
    SERIES = {Courant Lecture Notes in Mathematics},
    VOLUME = {10},
 PUBLISHER = {New York University, Courant Institute of Mathematical
              Sciences, New York; American Mathematical Society, Providence,
              RI},
      YEAR = {2003},
     PAGES = {xiv+323},
      ISBN = {0-8218-3399-5},
   MRCLASS = {35Q55 (35-01 35J10 35Q40)},
       DOI = {10.1090/cln/010},
       URL = {https://doi.org/10.1090/cln/010},
}

@article {CW92,
    AUTHOR = {Cazenave, T. and Weissler, F. B.},
     TITLE = {Rapidly decaying solutions of the nonlinear {S}chr\"{o}dinger
              equation},
   JOURNAL = {Comm. Math. Phys.},
  FJOURNAL = {Communications in Mathematical Physics},
    VOLUME = {147},
      YEAR = {1992},
    NUMBER = {1},
     PAGES = {75--100},
      ISSN = {0010-3616},
}

@article{JL22,
 author = {Jeanjean, L. and Lu, S.-S.},
 title = {On global minimizers for a mass constrained problem},
 fjournal = {Calculus of Variations and Partial Differential Equations},
 journal = {Calc. Var. Partial Differ. Equ.},
 issn = {0944-2669},
 volume = {61},
 number = {6},
 pages = {18},
 note = {Id/No 214},
 year = {2022},
 doi = {10.1007/s00526-022-02320-6},
 zbMATH = {7599248},
 Zbl = {1500.35155}
}

@article{Lions84,
 author = {Lions, P.-L.},
 title = {The concentration-compactness principle in the calculus of variations. {The} locally compact case. {I}},
 fjournal = {Annales de l'Institut Henri Poincar{\'e}. Analyse Non Lin{\'e}aire},
 journal = {Ann. Inst. Henri Poincar{\'e}, Anal. Non Lin{\'e}aire},
 issn = {0294-1449},
 volume = {1},
 pages = {109--145},
 year = {1984},
 doi = {10.1016/S0294-1449(16)30428-0},
 url = {https://eudml.org/doc/78069},
 zbMATH = {3859846},
 Zbl = {0541.49009}
}

@incollection {NO02,
    AUTHOR = {Nakanishi, K. and Ozawa, T.},
     TITLE = {Scattering problem for nonlinear {S}chr\"{o}dinger and {H}artree
              equations},
      NOTE = {Tosio Kato's method and principle for evolution equations in
              mathematical physics (Sapporo, 2001)},
BOOKTITLE = {Nonlinear Differential Equations and Applications },
SERIES = {Nonlinear Differential Equations and Applications },
   JOURNAL = {S\={u}rikaisekikenky\={u}sho K\B{o}ky\={u}roku},
  FJOURNAL = {S\={u}rikaisekikenky\={u}sho K\B{o}ky\={u}roku},
    NUMBER = {1234},
      YEAR = {2001},
     PAGES = {105--112},
PUBLISHER = {Springer},


}

@article {TY84,
    AUTHOR = {Tsutsumi, Y. and Yajima, K.},
     TITLE = {The asymptotic behavior of nonlinear {S}chr\"{o}dinger equations},
   JOURNAL = {Bull. Amer. Math. Soc. (N.S.)},
  FJOURNAL = {American Mathematical Society. Bulletin. New Series},
    VOLUME = {11},
      YEAR = {1984},
    NUMBER = {1},
     PAGES = {186--188},
      ISSN = {0273-0979},
   MRCLASS = {35Q20 (81C05)},
  
MRREVIEWER = {Woodford W. Zachary},
       DOI = {10.1090/S0273-0979-1984-15263-7},
       URL = {https://doi.org/10.1090/S0273-0979-1984-15263-7},
}

@article {T85,
    AUTHOR = {Tsutsumi, Y.},
     TITLE = {Scattering problem for nonlinear {S}chr\"{o}dinger equations},
   JOURNAL = {Ann. Inst. H. Poincar\'{e} Phys. Th\'{e}or.},
  FJOURNAL = {Annales de l'Institut Henri Poincar\'{e}. Physique Th\'{e}orique},
    VOLUME = {43},
      YEAR = {1985},
    NUMBER = {3},
     PAGES = {321--347},
      ISSN = {0246-0211},
   MRCLASS = {35Q20 (35B40 35P25 81F05)},
 
MRREVIEWER = {Hartmut Pecher},
       URL = {http://www.numdam.org/item?id=AIHPB_1985__43_3_321_0},
}

@article{Strauss74,
author="Strauss, W. A.",
editor="Lavita, J. A.
and Marchand, J.-P.",
title="Nonlinear Scattering Theory",
journal="Scattering Theory in Mathematical Physics",
year="1974",
publisher="Springer Netherlands",
address="Dordrecht",
pages="53--78",
}

@article{Satoshi15,
 author = {Masaki, S.},
 title = {A sharp scattering condition for focusing mass-subcritical nonlinear {Schr{\"o}dinger} equation},
 fjournal = {Communications on Pure and Applied Analysis},
 journal = {Commun. Pure Appl. Anal.},
 issn = {1534-0392},
 volume = {14},
 number = {4},
 pages = {1481--1531},
 year = {2015},
 doi = {10.3934/cpaa.2015.14.1481},
 zbMATH = {6451035},
 Zbl = {1320.35326}
}

@article{Satoshi17,
 author = {Masaki, {S.}},
 title = {On minimal nonscattering solution for focusing mass-subcritical nonlinear {Schr{\"o}dinger} equation},
 fjournal = {Communications in Partial Differential Equations},
 journal = {Commun. Partial Differ. Equations},
 issn = {0360-5302},
 volume = {42},
 number = {4},
 pages = {626--653},
 year = {2017},
 doi = {10.1080/03605302.2017.1286672},
 zbMATH = {6803166},
 Zbl = {1373.35289}
}

@article{GOV94,
 author = {Ginibre, J. and Ozawa, T. and Velo, G.},
 title = {On the existence of the wave operators for a class of nonlinear {Schr{\"o}dinger} equations},
 fjournal = {Annales de l'Institut Henri Poincar{\'e}. Physique Th{\'e}orique},
 journal = {Ann. Inst. Henri Poincar{\'e}, Phys. Th{\'e}or.},
 issn = {0246-0211},
 volume = {60},
 number = {2},
 pages = {211--239},
 year = {1994},
 url = {https://eudml.org/doc/76633},
 zbMATH = {563028},
 Zbl = {0808.35136}
}

@article {HN15,
    AUTHOR = {Hayashi, N. and Naumkin, P. I.},
     TITLE = {Scattering problem for the supercritical nonlinear
              {S}ch\"{o}dinger equation in 1d},
   JOURNAL = {Funkcial. Ekvac.},
  FJOURNAL = {Funkcialaj Ekvacioj. Serio Internacia},
    VOLUME = {58},
      YEAR = {2015},
    NUMBER = {3},
     PAGES = {451--470},
      ISSN = {0532-8721},
   MRCLASS = {35Q55 (35B40 35P25)},
  MRNUMBER = {3468737},
       DOI = {10.1619/fesi.58.451},
       URL = {https://doi-org.waseda.idm.oclc.org/10.1619/fesi.58.451},
}

@article{FH21,
 author = {Fukaya, N. and Hayashi, M.},
 title = {Instability of algebraic standing waves for nonlinear {Schr{\"o}dinger} equations with double power nonlinearities},
 fjournal = {Transactions of the American Mathematical Society},
 journal = {Trans. Am. Math. Soc.},
 issn = {0002-9947},
 volume = {374},
 number = {2},
 pages = {1421--1447},
 year = {2021},
doi = {10.1090/tran/8269},
 zbMATH = {7291903},
 Zbl = {1458.35387}
}

@article{LRN20,
 author = {Lewin, M. and Rota Nodari, S.},
 title = {The double-power nonlinear {Schr{\"o}dinger} equation and its generalizations: uniqueness, non-degeneracy and applications},
 fjournal = {Calculus of Variations and Partial Differential Equations},
 journal = {Calc. Var. Partial Differ. Equ.},
 issn = {0944-2669},
 volume = {59},
 number = {6},
 pages = {48},
 note = {Id/No 197},
 year = {2020},
 doi = {10.1007/s00526-020-01863-w},
}

@article{BDF24,
 author = {Bellazzini, J. and Dinh, V. D. and Forcella, L.},
 title = {Scattering for nonradial {3D} {NLS} with combined nonlinearities: the interaction {Morawetz} approach},
 fjournal = {SIAM Journal on Mathematical Analysis},
 journal = {SIAM J. Math. Anal.},
 issn = {0036-1410},
 volume = {56},
 number = {3},
 pages = {3110--3143},
 year = {2024},
 doi = {10.1137/23M1559063},
}

@article{BFG23,
 author = {Bellazzini, J. and Forcella, L. and Georgiev, V.},
 title = {Ground state energy threshold and blow-up for {NLS} with competing nonlinearities},
 fjournal = {Annali della Scuola Normale Superiore di Pisa. Classe di Scienze. Serie V},
 journal = {Ann. Sc. Norm. Super. Pisa, Cl. Sci. (5)},
 issn = {0391-173X},
 volume = {24},
 number = {2},
 pages = {955--988},
 year = {2023},
 doi = {10.2422/2036-2145.202005_044},
}

@article{SoaveJDE,
 author = {Soave, N.},
 title = {Normalized ground states for the {NLS} equation with combined nonlinearities},
 fjournal = {Journal of Differential Equations},
 journal = {J. Differ. Equations},
 issn = {0022-0396},
 volume = {269},
 number = {9},
 pages = {6941--6987},
 year = {2020},
 doi = {10.1016/j.jde.2020.05.016},
}

@article{SoaveJFA,
 author = {Soave, {N.}},
 title = {Normalized ground states for the {NLS} equation with combined nonlinearities: the {Sobolev} critical case},
 fjournal = {Journal of Functional Analysis},
 journal = {J. Funct. Anal.},
 issn = {0022-1236},
 volume = {279},
 number = {6},
 pages = {42},
 note = {Id/No 108610},
 year = {2020},
 doi = {10.1016/j.jfa.2020.108610},
}

@article{JJTV,
 author = {Jeanjean, L. and Jendrej, J. and Le, T. T. and Visciglia, N.},
 title = {Orbital stability of ground states for a {Sobolev} critical {Schr{\"o}dinger} equation},
 fjournal = {Journal de Math{\'e}matiques Pures et Appliqu{\'e}es. Neuvi{\`e}me S{\'e}rie},
 journal = {J. Math. Pures Appl. (9)},
 issn = {0021-7824},
 volume = {164},
 pages = {158--179},
 year = {2022},
 doi = {10.1016/j.matpur.2022.06.005},

}

@article{TVZ,
 author = {Tao, T. and Vi\c{s}an, M. and Zhang, X.},
 title = {The nonlinear {Schr{\"o}dinger} equation with combined power-type nonlinearities},
 fjournal = {Communications in Partial Differential Equations},
 journal = {Commun. Partial Differ. Equations},
 issn = {0360-5302},
 volume = {32},
 number = {8},
 pages = {1281--1343},
 year = {2007},
 doi = {10.1080/03605300701588805},
}

@article{Beg,
 author = {B{\'e}gout, P.},
 title = {Convergence to scattering states in the nonlinear {Schr{\"o}dinger} equation},
 fjournal = {Communications in Contemporary Mathematics},
 journal = {Commun. Contemp. Math.},
 issn = {0219-1997},
 volume = {3},
 number = {3},
 pages = {403--418},
 year = {2001},
 doi = {10.1142/S0219199701000421},
}

@article{BL83,
 author = {Berestycki, H. and Lions, P.-L.},
 title = {Nonlinear scalar field equations. {I}: {Existence} of a ground state},
 fjournal = {Archive for Rational Mechanics and Analysis},
 journal = {Arch. Ration. Mech. Anal.},
 issn = {0003-9527},
 volume = {82},
 pages = {313--345},
 year = {1983},
 doi = {10.1007/BF00250555},
}

@article{AIKN13,
 author = {Akahori, T. and Ibrahim, S. and Kikuchi, H and Nawa, H.},
 title = {Existence of a ground state and scattering for a nonlinear {Schr{\"o}dinger} equation with critical growth},
 fjournal = {Selecta Mathematica. New Series},
 journal = {Sel. Math., New Ser.},
 issn = {1022-1824},
 volume = {19},
 number = {2},
 pages = {545--609},
 year = {2013},
 doi = {10.1007/s00029-012-0103-5},
}

@article{Luo22,
 author = {Luo, Y.},
 title = {Sharp scattering for the cubic-quintic nonlinear {Schr{\"o}dinger} equation in the focusing-focusing regime},
 fjournal = {Journal of Functional Analysis},
 journal = {J. Funct. Anal.},
 issn = {0022-1236},
 volume = {283},
 number = {1},
 pages = {34},
 note = {Id/No 109489},
 year = {2022},
 doi = {10.1016/j.jfa.2022.109489},
}

@article{Luo24,
 author = {Luo, {Y.}},
 title = {On the sharp scattering threshold for the mass-energy double critical nonlinear {Schr{\"o}dinger} equation via double track profile decomposition},
 fjournal = {Annales de l'Institut Henri Poincar{\'e} C. Analyse Non Lin{\'e}aire},
 journal = {Ann. Inst. Henri Poincar{\'e} C, Anal. Non Lin{\'e}aire},
 issn = {0294-1449},
 volume = {41},
 number = {1},
 pages = {187--255},
 year = {2024},
 doi = {10.4171/AIHPC/71},
}

@article{Miao13,
 author = {Miao, C. and Xu, G. and Zhao, L.},
 title = {The dynamics of the 3D radial {NLS} with the combined terms},
 fjournal = {Communications in Mathematical Physics},
 journal = {Commun. Math. Phys.},
 issn = {0010-3616},
 volume = {318},
 number = {3},
 pages = {767--808},
 year = {2013},
 doi = {10.1007/s00220-013-1677-2},
}

@article{Miao16,
 author = {Miao, {C.} and Xu, G. and Zhao, L.},
 title = {The dynamics of the {NLS} with the combined terms in five and higher dimensions},
 journal = {Some topics in harmonic analysis and applications. Special volume dedicated to Shanzhen Lu on the occasion of his 75th birthday},
 isbn = {978-1-57146-315-9},
 pages = {265--298},
 year = {2016},
 publisher = {Somerville, MA: International Press; Beijing: Higher Education Press},

}

@article{Cheng16,
 author = {Cheng, X. and Miao, C. and Zhao, L.},
 title = {Global well-posedness and scattering for nonlinear {Schr{\"o}dinger} equations with combined nonlinearities in the radial case},
 fjournal = {Journal of Differential Equations},
 journal = {J. Differ. Equations},
 issn = {0022-0396},
 volume = {261},
 number = {6},
 pages = {2881--2934},
 year = {2016},

 doi = {10.1016/j.jde.2016.04.031},

}

@article{Cheng20,
 author = {Cheng, X.},
 title = {Scattering for the mass super-critical perturbations of the mass critical nonlinear {Schr{\"o}dinger} equations},
 fjournal = {Illinois Journal of Mathematics},
 journal = {Ill. J. Math.},
 issn = {0019-2082},
 volume = {64},
 number = {1},
 pages = {21--48},
 year = {2020},
 doi = {10.1215/00192082-8165582},

}

@article{AKN12,
 author = {Akahori, T. and Kikuchi, H. and Nawa, H.},
 title = {Scattering and blowup problems for a class of nonlinear {Schr{\"o}dinger} equations.},
 fjournal = {Differential and Integral Equations},
 journal = {Differ. Integral Equ.},
 issn = {0893-4983},
 volume = {25},
 number = {11-12},
 pages = {1075--1118},
 year = {2012},
}

@article{Lieb,
 author = {Lieb, Elliott H.},
 title = {On the lowest eigenvalue of the {Laplacian} for the intersection of two domains},
 fjournal = {Inventiones Mathematicae},
 journal = {Invent. Math.},
 issn = {0020-9910},
 volume = {74},
 pages = {441--448},
 year = {1983},

}

@article{IT-24-Vietnam,
 author = {Ifrim, M. and Tataru, D.},
 title = {Long time solutions for {1D} cubic dispersive equations. {II}: {The} focusing case},
 fjournal = {Vietnam Journal of Mathematics},
 journal = {Vietnam J. Math.},
 issn = {2305-221X},
 volume = {52},
 number = {3},
 pages = {597--614},
 year = {2024},
}

@article{IT-2025,
 author = {Ifrim, {M.} and Tataru, {D.}},
 title = {Global solutions for 1D cubic dispersive equations. {III}: {The} quasilinear {Schr{\"o}dinger} flow},
 fjournal = {Inventiones Mathematicae},
 journal = {Invent. Math.},
 issn = {0020-9910},
 volume = {242},
 number = {1},
 pages = {221--304},
 year = {2025},
}

@article{BL,
 author = {Br{\'e}zis, Ha{\"{\i}}m and Lieb, Elliott H.},
 title = {A relation between pointwise convergence of functions and convergence of functionals},
 fjournal = {Proceedings of the American Mathematical Society},
 journal = {Proc. Am. Math. Soc.},
 issn = {0002-9939},
 volume = {88},
 pages = {486--490},
 year = {1983}
}

\end{document}